\documentclass[12pt, reqno]{amsart}
\usepackage{mathrsfs}
\usepackage{graphicx}
\usepackage{tikz}
\usepackage{amsmath}
\usepackage{amsfonts}
\usepackage{amssymb}
\usepackage[pagebackref]{hyperref}
\usepackage{pifont}
\hypersetup{backref, pagebackref, colorlinks=true}
\usepackage{cite}
\usepackage{color}

\def\blue{\textcolor{blue}}

\def\magenta{\textcolor{magenta}}

\def\cyan{\textcolor{cyan}}

\usepackage{utfsym}

\usepackage{epic}

\newtheorem{thm}{Theorem}[section]

\newtheorem{lem}[thm]{Lemma}
\newtheorem{prop}[thm]{Proposition}

\newtheorem{conj}[thm]{Conjecture}
\newtheorem{remark}[thm]{Remark}

\newtheorem{?}[thm]{Problem}
\newtheorem{defi}{Definition}
\newtheorem{exam}{Example}
\newtheorem{Fact}[thm]{Fact}

\usepackage{amsthm,amsmath,amssymb,url,cite,color}

\def\blue{\textcolor{blue}}
\def\magenta{\textcolor{magenta}}

\numberwithin{equation}{section}

\def\rmida{\mathrm{rmida}}

\def\lmi{\mathrm{lmi}}
\def\rmi{\mathrm{rmi}}
\def\lmidd{\mathrm{lmidd}}
\def\rmida{\mathrm{rmida}}
\def\lrmi{\mathrm{lrmi}}

\def\pk{\mathrm{pk}}

\def\drop{\mathrm{drop}}

\def\des{\mathrm{des}}

\def\lmin{\mathrm{lmi}}
\def\rmin{\mathrm{rmi}}

\def\S{\mathfrak{S}}

\def\mix{\mathrm{mix}}
\def\max{\mathrm{max}}
\def\fix{\mathrm{fix}}
\def\cyc{\mathrm{cyc}}
\def\mix{\mathrm{mix}}
\def\Web{\mathrm{Web}}
\def\Orb{\mathrm{Orb}}

\def\a{\alpha}
\def\b{\beta}

\def\Z{\mathbb Z}

\def\cyc{\mathrm{cyc}}

\def\Cr{\mathrm{Cr}}

\newcommand\Xmarking[2]
{\draw[very thick] (#1-0.7,#2-0.7)--(#1-0.3,#2-0.3);
\draw[very thick] (#1-0.7,#2-0.3)--(#1-0.3,#2-0.7);
\draw (#1-1,#2-0.5)--(#1-0.5,#2-0.5);
\draw (#1-0.5,#2-0.5)--(#1-0.5,#2); }
\newcommand\Cross[2]
{\draw (#1-0.5,#2)--(#1-0.5,#2-1);
\draw (#1,#2-0.5)--(#1-1,#2-0.5); }
\newcommand\UP[2]{\draw (#1-0.5,#2)--(#1-0.5,#2-1);}
\newcommand\EAST[2]{\draw (#1,#2-0.5)--(#1-1,#2-0.5);}
\newcommand\Asmooth[2]
{\draw (#1,#2-0.5) .. controls (#1-0.45,#2-0.45) and (#1-0.45,#2-0.45) .. (#1-0.5,#2);
\draw (#1-1,#2-0.5) .. controls (#1-0.55,#2-0.55) and (#1-0.55,#2-0.55) .. (#1-0.5,#2-1); }

\newcommand\Matching[2]{\draw (#1,0) to [out=55,in=125] (#2,0);}

\begin{document}

\title[Web permutations, Seidel triangle and normalized $\gamma$-coefficients]{Web permutations, Seidel triangle and \\normalized $\gamma$-coefficients}

\author[Y. Dong]{Yao Dong}
\address[Yao Dong]{Research Center for Mathematics and Interdisciplinary Sciences, Shandong University \& Frontiers Science Center for Nonlinear Expectations, Ministry of Education, Qingdao 266237, P.R. China}
\email{y.dong@mail.sdu.edu.cn}

\author[Z. Lin]{Zhicong Lin}
\address[Zhicong Lin]{Research Center for Mathematics and Interdisciplinary Sciences, Shandong University \& Frontiers Science Center for Nonlinear Expectations, Ministry of Education, Qingdao 266237, P.R. China}
\email{linz@sdu.edu.cn}

\author[Q.Q. Pan]{Qiongqiong Pan}
\address[Qiongqiong Pan]{College of Mathematics and Physics, Wenzhou University, Wenzhou 325035, P.R. China}
\email{qpan@wzu.edu.cn}

\date{\today}
\begin{abstract}
The web permutations were introduced by Hwang, Jang and Oh to interpret the entries of the transition matrix between the Specht and $\mathrm{SL}_2$-web bases of the irreducible $\S_{2n}$-representation indexed by $(n,n)$. They conjectured that certain classes of web permutations are enumerated by the Seidel triangle. Using generating functions, Xu and Zeng showed that enumerating web permutations by the number of drops, fixed points and cycles gives rise to the normalized $\gamma$-coefficients of the $(\alpha,t)$-Eulerian  polynomials. They posed the problems to prove their result combinatorially  and to find an interpretation of the normalized $\gamma$-coefficients in terms of cycle-up-down  permutations. In this work, we prove the enumerative conjecture of Hwang-Jang-Oh and answer the two open problems proposed by Xu and Zeng.

\end{abstract}

\keywords{Web permutations; Andr\'e permutations; Seidel triangle; group actions; permutation statistics.}
\maketitle

\section{Introduction}
Let $[n]:=\{1,2, \ldots, n\}$ and $\S_{n}$ be the set of permutations of $[n]$.
In order to interpret the entries of the transition matrix between the Specht and $\mathrm{SL}_2$-web bases of the irreducible $\S_{2n}$-representation indexed by $(n,n)$, Hwang, Jang and Oh~\cite{HJO} introduced the web permutations. These permutations are defined by applying smoothing and switching operations on grid configurations until there is no crossing (see Section~\ref{sec:webper}). 
One of the main results proved in~\cite{HJO} is a nice characterization of web permutations closely related to the Andr\'e permutations invented by Foata and Sch\"utzenberger~\cite{SF73}. 

The  {\em Andr\'e permutations} (of the first kind) can be defined recursively as follows. First, the empty word is considered to be an Andr\'e permutation. Given a permutation $w=w_1w_2\cdots w_n$ ($n\geq1$) of distinct positive integers with the smallest letter $w_k$, then $w$ is an Andr\'e permutation if both $w_1w_2\cdots w_{k-1}$ and $w_{k+1}w_{k+2}\cdots w_n$ are  Andr\'e permutations and $\max\{w_1,\ldots,w_{n}\}=\max\{w_k,\ldots,w_n\}$. For instance, $31245$ is an Andr\'e permutation but $132$ is not. 
Let $C=(a_1,\ldots,a_k)$ be a cycle (cyclic permutation)  with $a_1=\min\{a_1,\ldots,a_k\}$. Then, cycle $C$ is called an \emph{Andr\'e cycle} if the word $a_2\cdots a_k$ is an Andr\'e permutation. For instance, $C=(2,3,5,6,1,4)$ is an Andr\'e cycle since $C=(1,4,2,3,5,6)$ and the permutation $42356$ is an Andr\'e permutation. 

Note that any permutation  can be written  as a product of disjoint cycles. The following characterization of web permutations is due to Hwang, Jang and Oh~\cite[Theorem~3.4]{HJO}. 

\begin{thm}[Hwang, Jang and Oh]\label{thm:huang}
A permutation $\sigma\in\S_n$ is a web permutation iff each cycle of $\sigma$ is an Andr\'e cycle. 
\end{thm}

Let $\Web_n$ be the set of web permutations in $\S_n$. Based on the above characterization, they~\cite{HJO} investigated the enumerative properties of web permutations. In particular, they showed that counting permutations  $\sigma\in\Web_n$ with first letter $\sigma_1=n+1-k$ gives the {\em Entringer number} $E_{n,k}$, a refinement of the classical Euler numbers. At the end of their article, they also considered the permutations $\sigma\in\Web_n$ whose associated matching $M(\sigma)$ equals $M_0^{(n)}$, where $M(\sigma)$ is defined in Section~\ref{sec:webper} and 
\begin{equation}\label{eq:Mat0}
M_0^{(n)}:=\{\{0,1\},\{2,3\},\ldots,\{2n-2,2n-1\}\}
\end{equation} 
is the unique matching of $\{0,1,\ldots,2n-1\}$ that is simultaneously noncrossing and nonnesting. Let $\widetilde\Web_n:=\{\sigma\in\Web_n: M(\sigma)=M_0^{(n)}\}$ and let $f(n,k)$ be the number of $\sigma\in\widetilde\Web_n$ with $\sigma_1=k$. They showed that $f(n,2k)=f(n,n)=0$ (for $n\ge 1$ and $1\leq k \leq \lfloor n/2\rfloor$) and conjectured the following new interpretation for the {\em Seidel triangle}~\cite{Sei77} $(s_{i,j})_{i,j\geq1}$
such that $s_{1,1}=s_{2,1}=1$ and 
$$\begin{cases}
s_{i,j}= 0,\qquad\qquad\qquad\qquad&\text{ if } j=0 \mbox{ or} ~j>\lceil i/2 \rceil;\\
    s_{2i+1,j}= s_{2i+1,j-1} + s_{2i,j},  &\text{ if } j=1,2,\dots, i+1;\\
    s_{2i,j}= s_{2i,j+1} + s_{2i-1,j},  &\text{ if} ~j=1,2,\dots,i.
  \end{cases}$$

\begin{conj}[\text{Hwang, Jang and Oh~\cite[Conjecture~4.9]{HJO}}]\label{conj:hwang}
For $n\geq1$, we have
$$\begin{cases}
f(2n-1,2k-1)=s_{2n-2,k}, \\
f(2n,2k-1)= s_{2n-1,n-k+1}.
\end{cases}$$
\end{conj}

\begin{figure}
  \begin{tabular}{c|ccccc}
    \( i\backslash j \) & 1 & 2 & 3 & 4 & 5 \\ \hline
    1 & 1 & & & & \\
    2 &1 & & & & \\
    3 & 1 & 1 & & & \\
    4 & 2 & 1 & & & \\
    5 & 2 & 3 & 3 & & \\
    6 & 8 & 6 & 3 & & \\ 
    7 & 8 & 14 & 17 & 17 & \\
    8 & 56 & 48 & 34 & 17 &\\
    9 & 56 & 104 & 138 & 155 & 155\\
  \end{tabular}
    \caption{The Seidel triangle $s_{i,j}$ for $i\leq 9$ and $j\leq 5$.}
    \label{seidel}
 \end{figure}
 
 The first values of the Seidel triangle are given as in Fig.~\ref{seidel}. Note that the Seidel triangle  refines the {\em Genocchi numbers of the first kind} $g_{2n-1}$ and the {\em median Genocchi numbers} $g_{2n}$ (also known as {\em Genocchi numbers of the second kind}) by 
 $$ g_{2n-1}=s_{2n-1,n}, \quad g_{2n}=s_{2n,1}.$$ 
 The Seidel triangle and Genocchi numbers appear frequently in combinatorics and other domains of mathematics~\cite{Ang,Dum,Fei,FLS,HJO,PZ2,RE,TN}. In this work, we connect web permutations with the expansions of  chord diagrams studied by Nakamigawa~\cite{TN16,TN}, leading to a proof of Conjecture~\ref{conj:hwang}. 
 
 Recent work of Xu and Zeng~\cite{Xu} shows that web permutations can also be used to interpret the normalized $\gamma$-coefficients of the $(\alpha,t)$-Eulerian  polynomials. For $\sigma=\sigma_1\cdots\sigma_n\in\S_n$, an index $i$ is called  a {\em descent} (resp., {\em drop}) of $\sigma$ if $\sigma_i>\sigma_{i+1}$ (resp.,~$i>\sigma_i$). It is well known (see~\cite[Section~1.4]{St}) that the classical {\em Eulerian polynomial} $A_n(x)$ has the following interpretations 
 $$
 A_n(x)=\sum_{\sigma\in\S_n} x^{\des(\sigma)}=\sum_{\sigma\in\S_n} x^{\drop(\sigma)},
 $$
 where $\des(\sigma)$ (resp.,~$\drop(\sigma)$) is the number of descents (resp.,~drops) of $\sigma$. 
 As was proved by  Foata and Sch\"utzenberger~\cite{FS-70} (see also~\cite{PZ}), the Eulerian polynomial $A_n(x)$ admits the following $\gamma$-positivity expansion
 \begin{equation}\label{divi:gam}
 A_n(x)=\sum_{i=0}^{\lfloor(n-1)/2\rfloor}2^id_{n,i}x^i(1+x)^{n-1-2i},
 \end{equation}
 where $d_{n,i}$ is the number of Andr\'e permutations in $\S_n$ with $i$ descents. 
Given $\sigma\in\S_n$, the letter $\sigma_i$ is a {\em left-to-right minimum} (resp., {\em right-to-left minimum}) of $\sigma$ if $\sigma_i<\sigma_j$  for all $j<i$ (resp.,~$j>i$).  With the convention $\sigma_0=\sigma_{n+1}=+\infty$,  the letter $\sigma_i$ is  called a {\em double descent} (resp., {\em double ascent, peak, valley}) of $\sigma$ if $\sigma_{i-1}>\sigma_i>\sigma_{i+1}$ (resp.,~$\sigma_{i-1}<\sigma_i<\sigma_{i+1}$, $\sigma_{i-1}<\sigma_i>\sigma_{i+1}$, $\sigma_{i-1}>\sigma_i<\sigma_{i+1}$). Furthermore, if   $\sigma_i$  is not only a left-to-right minimum (resp.,~right-to-left minimum) but also a double descent (resp.,~double ascent), then it is called a {\em left-to-right-minimum-double-descent} (resp.,~{\em right-to-left-minimum-double-ascent}). 
 Inspired by the works in~\cite{CS,CS77,Ji23,JL,MQYY}, Xu and Zeng~\cite{Xu} considered the {\em$(\alpha,t)$-Eulerian  polynomials}
$$
  A_n(x,t|\alpha):=\sum_{\sigma  \in \S_n}x^{\des(\sigma )}t^{\lmidd(\sigma )+\rmida(\sigma)}\alpha^{\lmi(\sigma)+\rmi(\sigma)-2},
$$
where $\lmi(\sigma)$  (resp.,~$\rmi(\sigma)$, $\lmidd(\sigma)$, $\rmida(\sigma)$) denotes the number of left-to-right minima (resp.,~right-to-left minima, left-to-right-minimum-double-descents, right-to-left-minimum-double-ascents)  of $\sigma$. 
 In~\cite[Theorem~2.6]{Xu}, they proved the following $\gamma$-positivity expansion of $A_n(x,t|\alpha)$:
\begin{equation}\label{gam:at}
A_n(x,t|\alpha)=\sum_{i=0}^{\lfloor(n-1)/2\rfloor}\gamma_{n,i}(\alpha,t)x^i(1+x)^{n-1-2i}
\end{equation}
 with 
 \begin{equation}
 \gamma_{n,i}(\alpha,t)=\sum_{\sigma\in\widetilde\S_n\atop{\des(\sigma)=i}}\alpha^{\lmi(\sigma)+\rmi(\sigma)-2}t^{\rmida(\sigma)},
 \end{equation}
 where $\widetilde\S_n$ is the set of permutations in $\S_n$ with no double descents. 
It should be noted that the $t=1$ case of the $\gamma$-positivity expansion~\eqref{gam:at} was first proved by Ji in~\cite{Ji23} via Chen's context-free grammar and latter proved combinatorially by Ji and Lin~\cite{JL} via introducing a new group action on permutations. 
 
 Using generating functions, Xu and Zeng~\cite[Theorem~2.9]{Xu} further proved that 
 
 \begin{thm}[\text{Xu and Zeng}] For $0\leq i\leq\lfloor (n-1)/2\rfloor$, we have 
 \begin{equation}\label{divi:gam:at}
  \gamma_{n,i}(\alpha,t)=2^id_{n,i}(\alpha,t)
 \end{equation}
 with 
 \begin{equation}\label{int:dni}
 d_{n,i}(\alpha,t)=\sum_{\sigma\in\Web_{n-1}\atop{\drop(\sigma)=i}}t^{\fix(\sigma)}\alpha^{\cyc(\sigma)},
 \end{equation}
 where $\fix(\sigma)$ (resp.,~$\cyc(\sigma)$) denotes the number of fixed points (resp., cycles) of $\sigma$.
 \end{thm}
 The polynomials $\gamma_{n,i}(\alpha,t)$ and $d_{n,i}(\alpha,t)$ are called the {\em$\gamma$-coefficients} and  the {\em normalized $\gamma$-coefficients} of the $(\alpha,t)$-Eulerian  polynomials, respectively.
At the end of~\cite{Xu},  they also posed the following open problem for further research. 
 \begin{?}[\text{Xu and Zeng~\cite[Problem~2]{Xu}}]\label{pro1:xz}
 Find a combinatorial proof of~\eqref{divi:gam:at}. 
 \end{?}
 
We shall  construct a group action on permutations without double descents and prove~\eqref{divi:gam:at} combinatorially, thereby answering Problem~\ref{pro1:xz}.

 A permutation $\sigma\in\S_n$ is an {\em up-down} permutation if 
 $$
 \sigma_1<\sigma_2>\sigma_3<\cdots.
 $$
 Let $C=(a_1,\ldots,a_k)$ be a cycle (cyclic permutation)  with $a_1=\min\{a_1,\ldots,a_k\}$.  
Then, the  cycle $C$ is said to be \emph{up-down} if  $a_1<a_2>a_3<\cdots$, i.e., $a_1a_2\cdots a_k$ is an up-down permutation. We say that a permutation $\sigma$ is \emph{cycle-up-down} if it is a product of disjoint up-down cycles. 
Let $\Delta_n$ be the set of cycle-up-down permutations in $\S_n$. Cycle-up-down permutations were considered by Deutsch and Elizade~\cite{DE11}, where they proved the following bivariate generating function
\begin{equation}\label{eq:Eli}
\sum_{n\geq0}\sum_{\sigma\in\Delta_n}t^{\mathrm{fix}(\sigma)}\alpha^{\mathrm{cyc}(\sigma)}\frac{z^n}{n!}=\frac{e^{\alpha(t-1)z}}{(1-\sin z)^{\alpha}}.
\end{equation}
On the other hand, Xu and Zeng~\cite{Xu} proved that 
\begin{equation*}\label{eq:XuZ}
\sum_{n\geq0}\sum_{\sigma\in\Web_n}t^{\fix(\sigma)}\alpha^{\cyc(\sigma)}\frac{z^n}{n!}=\frac{e^{\alpha(t-1)z}}{(1-\sin z)^{\alpha}}.
\end{equation*}
Comparing with~\eqref{eq:Eli} gives the equidistribution 
\begin{equation}\label{equi:XuZ}
\sum_{\sigma\in\Delta_n}t^{\mathrm{fix}(\sigma)}\alpha^{\mathrm{cyc}(\sigma)}=\sum_{\sigma\in\Web_n}t^{\fix(\sigma)}\alpha^{\cyc(\sigma)}. 
\end{equation}
As was noted in~\cite{Xu}, Foata and Han~\cite{F-H14} constructed a bijection between Andr\'e permutations  and up-down permutations, which preserves the size and the first letter.  Applying the bijection to each cycle of $\sigma\in\Web_n$ immediately yields a bijection  from $\Web_n$ to $\Delta_n$ that proves~\eqref{equi:XuZ}. In view of~\eqref{int:dni} and~\eqref{equi:XuZ}, Xu and Zeng~\cite{Xu} also posed the following interesting open problem. 

\begin{?}[\text{Xu and Zeng~\cite[Problem~1]{Xu}}]\label{pro2:xz}
Find a suitable statistic $\widehat{\mathrm{drop}}$ over $\Delta_n$ and a bijection $\Lambda:\Web_n\rightarrow\Delta_n$ such that $\mathrm{drop}(\sigma)=\widehat{\mathrm{drop}}(\Lambda(\sigma))$ for $\sigma\in\Web_n$.
\end{?}

 We answer the above open problem by introducing a recursively defined statistic ``$\mix$" on permutations. First, define $\mathrm{mix}(\emptyset)=0$. 
 For $\sigma=\sigma_1\sigma_2\cdots\sigma_n\in\S_n$ with $n\geq1$, suppose that  $j$ is the least integer for which either  $\sigma_j=\min\{\sigma_1,\ldots,\sigma_n\}$ or $\sigma_j=\max\{\sigma_1,\ldots,\sigma_n\}$. Then define
$$
\mathrm{mix}(\sigma):=
\begin{cases}
0,&\text{if $n=1$};\\
\mathrm{mix}(\sigma_1\cdots\sigma_{j-1})+\mathrm{mix}(\sigma_{j+1}\cdots\sigma_{n})+\chi(1<j<n),& \text{otherwise}.
\end{cases}
$$
Here $\chi(\mathsf{S})$ equals $1$, if the statement $\mathsf{S}$ is true; and $0$, otherwise.
For example,  If $\sigma=58264713\in\S_8$, then by searching for  more forward positions of the smallest/largest letters  successively, we have
$$
\mix(\sigma)=1+\mix(5)+\mix(264713)=2+\mix(264)+\mix(13)=2.
$$
Let $C=(a_1,\ldots,a_l)$ be a cycle (cyclic permutation)  with $a_1=\min\{a_1,\ldots,a_l\}$. Define 
\begin{equation}\label{def:drop}
\widehat{\mathrm{drop}}(C)=\mix(a_1a_2\cdots a_l)+\chi(l>1)\quad\text{and}\quad\widehat{\mathrm{drop}}(\sigma)=\sum_{i=1}^k\widehat{\mathrm{drop}}(C_i),
\end{equation}
assuming $\sigma=C_1C_2\cdots C_k\in\S_n$ (written in product of cycles).

The following result answers Problem~\ref{pro2:xz}.
\begin{thm}\label{thm:pro2xz}
There exists a bijection $\Lambda:\Web_n\rightarrow\Delta_n$ such that $\mathrm{drop}(\sigma)=\widehat{\mathrm{drop}}(\Lambda(\sigma))$ for $\sigma\in\Web_n$. Consequently,
\begin{equation}\label{int:dni2}
 d_{n+1,i}(\alpha,t)=\sum_{\sigma\in\Delta_{n}\atop{\widehat\drop(\sigma)=i}}t^{\fix(\sigma)}\alpha^{\cyc(\sigma)}.
 \end{equation}
\end{thm}
\begin{remark}
It is interesting to note that $\widehat{\mathrm{drop}}(\sigma)=\mathrm{drop}(\sigma)$ for any $\sigma\in\Web_n$.
\end{remark}
  
The rest of this article is organized as follows. In Section~\ref{Sec:conj}, we review the original definition of web permutations 
and prove Conjecture~\ref{conj:hwang} using some known results on the chord diagrams closely related to the nature of the web permutations. In Section~\ref{Sec:groact}, we introduce a new group action on permutations which enables us to provide a combinatorial proof of~\eqref{divi:gam:at}.
In Section~\ref{Sec:pro2}, we first review  the min-max trees and the Hetyei--Reiner action, and then prove Theorem~\ref{thm:pro2xz}.

\section{Expansions of  chord diagrams and proof of conjecture~\ref{conj:hwang}}
\label{Sec:conj}
In this section, we establish a connection between chord diagrams and web permutations, leading to a proof of Conjecture~\ref{conj:hwang}. 

\subsection{Web permutations}
\label{sec:webper}

We need to review the original definition of Web permutations introduced in~\cite{HJO}. 
Given an $n\times n$ cell chart $G_n$ in the first quadrant of the $xy$-plane such that the lower left coordinate is $(0,0)$ and the upper right coordinate is $(n,n)$. 
A cell in $G_n$ is recorded by $(i,j)$ if the $x$-coordinate and $y$-coordinate of its upper-right corner are $i$ and $j$, respectively. 
Given a permutation $\sigma=\sigma_1\cdots\sigma_n\in\S_n$, we first label the cell $(i,\sigma_i)$ of $G_n$ by a symbol  $\Small{\usym{2715}}$ for each $i\in[n]$,   
 and then take each $\Small{\usym{2715}}$ as a starting point, draw a straight line left and draw a straight line up to the boundary of $G_n$. 
The resulting cell chart is called the \emph{empty grid configuration} of $\sigma$; see Fig.~\ref{elbow}~$(a)$.

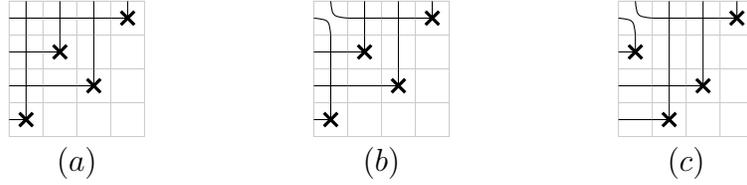
\begin{figure}
  \centering
  \begin{tikzpicture}[scale=0.45]
     
  \draw[black!20] (0,-2) grid (4,2);
  \Xmarking{4}{2}\Xmarking{2}{1}\Xmarking{3}{0}\Xmarking{1}{-1}
  \Cross{1}{0} \Cross{1}{1} \Cross{1}{2}\Cross{2}{2}
  \UP{3}{1}
  \EAST{2}{0}
  \Cross{3}{2}
  \node[below] at (2,-2) {$(a)$};

  \draw[black!20] (9,-2) grid (13,2);
  \Xmarking{13}{2}\Xmarking{11}{1}\Xmarking{12}{0}\Xmarking{10}{-1}
  \Cross{10}{0} \Cross{10}{1}
  \UP{12}{1}
  \EAST{11}{0}
  \Asmooth{10}{2} \Cross{11}{2}
  \Cross{12}{2}
  \node[below] at (11,-2) {$(b)$};

  \draw[black!20] (18,-2) grid (22,2);
  \Xmarking{22}{2}\Xmarking{20}{-1}\Xmarking{21}{0}\Xmarking{19}{1}
  \UP{21}{1} \UP{20}{1} \UP{20}{0}
  \EAST{20}{0} \EAST{19}{0} \EAST{19}{-1}
  \Asmooth{19}{2} \Cross{20}{2}
  \Cross{21}{2}
  \node[below] at (20,-2) {$(c)$};
  \end{tikzpicture}
  \caption{$(a)$, $(b)$ and $(c)$ represent the grid configurations $G(1324,\emptyset)$, $G(1324,\{(1,4)\})$ and $G(3124,\{(1,4)\})$, respectively.}
  \label{elbow}
\end{figure}

A cell $(i,j)$ in the empty grid configuration is called a {\em crossing} of $\sigma$ if there are both a horizontal line and a vertical line passing through it, i.e., $i<\sigma_j^{-1}$ and $j>\sigma_i$. 
 Let $\Cr(\sigma)$ be the set of all crossings of $\sigma$.  For any subset $E\subseteq \Cr(\sigma)$, let $G(\sigma,E)$ be the \emph{grid configuration}  obtained from the empty configuration by replacing each crossing in $E$ with an elbow (see Fig.~\ref{elb}); see Fig.~\ref{elbow}~$(b)$ for $G(1324,\{(1,4)\})$. Notice that now $\Cr(\sigma)\setminus E$ is the set of all crossings of $G(\sigma,E)$ and the empty grid configuration of $\sigma$ is $G(\sigma,\emptyset)$.
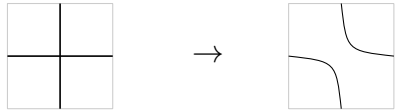
\begin{figure}[h]
      \begin{tikzpicture}[scale=0.7]
          \draw[black!20] (0,0) grid (2,2);
          \draw[line width=0.6pt] (0,1) - - (2,1);
          \draw[line width=0.6pt] (1,0) - - (1,2);
      \end{tikzpicture}
    \begin{tikzpicture}[scale=0.7]
          \draw[white!20] (0,0) grid (1,3);
          \node[text width=1cm] at (2, 1) {$\to$};
          \end{tikzpicture}
      \begin{tikzpicture}[scale=1.4]
          \draw[black!20] (0,0) grid (1,1);
          \Asmooth{1}{1}
           \end{tikzpicture}
         \caption{A crossing to an elbow.\label{elb}}
      \end{figure}

In order to introduce the algorithm for resolving the crossings, we 
define a partial order on crossings  by $(i,j)\succeq (i',j')$ if $i\leq i'$ and $j\geq j'$, i.e., $(i,j)$ appears weakly northwest of $(i',j')$.

\begin{defi}
  Let $\sigma=\sigma_1\cdots\sigma_n\in\S_n$ and $E\subseteq \Cr(\sigma)$.
  Let $(i,j)$ be a maximal crossing in $G(\sigma,E)$ that is made by two symbols $\Small{\usym{2715}}$ at positions $(i,\sigma_i)$ and $(\sigma^{-1}_j,j)$. 
  An operation that resolving  the crossing $(i,j)$ is called
\begin{itemize}
\item {\bf smoothing} if we replace $(i,j)$ with an elbow;
\item {\bf switching} if we  move the symbols $\Small{\usym{2715}}$ at cells $(i,\sigma_i), (\sigma^{-1}_j,j)$ to cells $(i,j), (\sigma^{-1}_j,\sigma_i)$,  and then delete the two straight lines connecting the  cells $(i,\sigma_i),(\sigma^{-1}_j,j)$ with $(i,j)$ but draw  two 
straight lines connecting the cells $(i,\sigma_i),(\sigma^{-1}_j,j)$ with $(\sigma^{-1}_j,\sigma_i)$.
\end{itemize}
\end{defi}

For example, resolving   the maximal crossing $(1,4)$ in Fig.~\ref{elbow}~$(a)$ by smoothing gives Fig.~\ref{elbow}~$(b)$, while 
 resolving  the  maximal crossing $(1,3)$ in Fig.~\ref{elbow}~$(b)$ by switching gives Fig.~\ref{elbow}~$(c)$.

Let $\sigma\in\S_n$, $E\subsetneq \Cr(\sigma)$ and $c=(i,j)$ be a maximal crossing in $\Cr(\sigma)\setminus E$. 
Suppose that there is no crossing in $E$ that is less than $c$. Then resolving  the crossing $c$ by smoothing  and witching  give respectively  $G(\sigma,E\cup\{c\})$ and $G(\sigma',E)$ with $\sigma'=\sigma\cdot t_{i,\ell}$, where $\ell=\sigma^{-1}_j$ and $t_{i,\ell}$ is the transposition which swaps $i$ and $\ell$. 
We write 
\begin{equation}\label{G()}
G(\sigma,E)=G(\sigma,E\cup\{c\})+G(\sigma',E).
\end{equation}

 Let $id_n$ be the identity permutation in $\S_n$. Starting from $G(id_n, \emptyset)$, we obtain two grid configurations, $G(id_n,\{(1,n)\})$ and $G( id_n\cdot t_{1,n},\emptyset)$,  by resolving the maximal crossing $(1,n)$  by smoothing and switching,  respectively.
By resolving (maximal) crossings until there are no crossings left, we get configurations of the form   $G(\sigma, \Cr(\sigma))$. For any such remaining configuration $G(\sigma, \Cr(\sigma))$, the permutation $\sigma\in\S_n$ is called a {\bf web permutation}. 
Basing on~\eqref{G()} and Proposition~\ref{web:uniq}, we get the following expression in terms of grid configuration:
$$
G(id_n, \emptyset)=\sum_{\sigma\in\Web_n}G(\sigma,\Cr(\sigma)). 
$$
See Fig.~\ref{webperm} for the process of resolving maximal crossings when $n=3$, from which we get 
$\Web_3=\{123, 132, 213, 231, 321\}$.

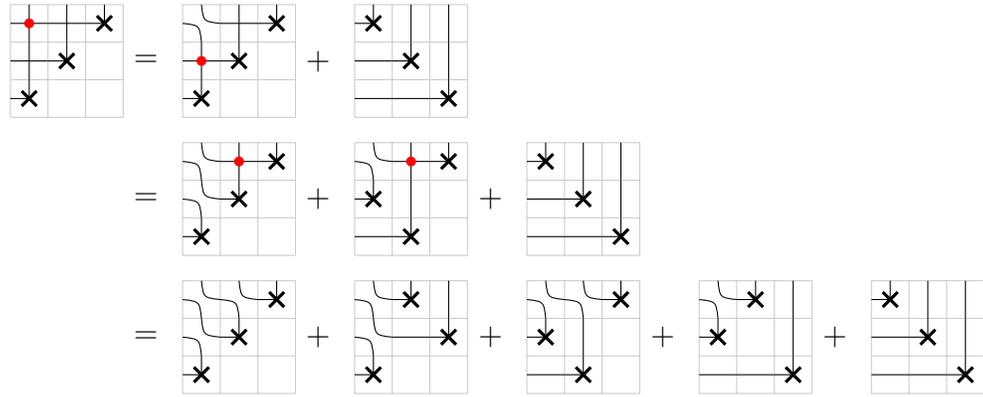
\begin{figure}
\begin{align*}
    \begin{tikzpicture}[scale=0.5]
        \draw[black!20] (0,0) grid (3,3);
        \Xmarking{1}{1}\Xmarking{2}{2}\Xmarking{3}{3}
        \Cross{1}{2}\Cross{1}{3}\Cross{2}{3}
            \node[circle, fill=red, draw=red, scale=0.3] at (0.5,2.5) {};
    \end{tikzpicture}
  &\begin{tikzpicture}[scale=0.5]
        \draw[white!20] (0,0) grid (1,3);
        \node[text width=0.5cm] at (0.5, 1.5) {$=$};
        \end{tikzpicture}
    \begin{tikzpicture}[scale=0.5]
        \draw[black!20] (0,0) grid (3,3);
        \Xmarking{1}{1}\Xmarking{2}{2}\Xmarking{3}{3}
        \Cross{1}{2}\Asmooth{1}{3}\Cross{2}{3}
            \node[circle, fill=red, draw=red, scale=0.3] at (0.5,1.5) {};
    \end{tikzpicture}
    \begin{tikzpicture}[scale=0.5]
        \draw[white!20] (0,0) grid (1,3);
        \node[text width=0.5cm] at (0.5, 1.5) {$+$};
        \end{tikzpicture}
    \begin{tikzpicture}[scale=0.5]
        \draw[black!20] (0,0) grid (3,3);
        \Xmarking{1}{3}\Xmarking{2}{2}\Xmarking{3}{1}
        \UP{2}{3}\EAST{1}{2}\UP{3}{2}\UP{3}{3}\EAST{1}{1}\EAST{2}{1}
    \end{tikzpicture}\\
    &\begin{tikzpicture}[scale=0.5]
        \draw[white!20] (0,0) grid (1,3);
        \node[text width=0.5cm] at (0.5, 1.5) {$=$};
        \end{tikzpicture}\begin{tikzpicture}[scale=0.5]
        \draw[black!20] (0,0) grid (3,3);
        \Xmarking{1}{1}\Xmarking{2}{2}\Xmarking{3}{3}
        \Cross{2}{3}\Asmooth{1}{3}\Asmooth{1}{2}
            \node[circle, fill=red, draw=red, scale=0.3] at (1.5,2.5) {};
    \end{tikzpicture}
        \begin{tikzpicture}[scale=0.5]
        \draw[white!20] (0,0) grid (1,3);
        \node[text width=0.5cm] at (0.5, 1.5) {$+$};
        \end{tikzpicture}
            \begin{tikzpicture}[scale=0.5]
        \draw[black!20] (0,0) grid (3,3);
        \Xmarking{1}{2}\Xmarking{2}{1}\Xmarking{3}{3}
        \Cross{2}{3}\Asmooth{1}{3}\EAST{1}{1}\UP{2}{2}
                    \node[circle, fill=red, draw=red, scale=0.3] at (1.5,2.5) {};
    \end{tikzpicture}
        \begin{tikzpicture}[scale=0.5]
        \draw[white!20] (0,0) grid (1,3);
        \node[text width=0.5cm] at (0.5, 1.5) {$+$};
        \end{tikzpicture}
    \begin{tikzpicture}[scale=0.5]
        \draw[black!20] (0,0) grid (3,3);
        \Xmarking{1}{3}\Xmarking{2}{2}\Xmarking{3}{1}
        \UP{2}{3}\EAST{1}{2}\UP{3}{2}\UP{3}{3}\EAST{1}{1}\EAST{2}{1}
    \end{tikzpicture}\\
    &\begin{tikzpicture}[scale=0.5]
        \draw[white!20] (0,0) grid (1,3);
        \node[text width=0.5cm] at (0.5, 1.5) {$=$};
        \end{tikzpicture}    \begin{tikzpicture}[scale=0.5]
        \draw[black!20] (0,0) grid (3,3);
        \Xmarking{1}{1}\Xmarking{2}{2}\Xmarking{3}{3}
        \Asmooth{1}{2}\Asmooth{1}{3}\Asmooth{2}{3}
    \end{tikzpicture}
        \begin{tikzpicture}[scale=0.5]
        \draw[white!20] (0,0) grid (1,3);
        \node[text width=0.5cm] at (0.5, 1.5) {$+$};
        \end{tikzpicture}
         \begin{tikzpicture}[scale=0.5]
        \draw[black!20] (0,0) grid (3,3);
        \Xmarking{1}{1}\Xmarking{2}{3}\Xmarking{3}{2}
        \Asmooth{1}{2}\Asmooth{1}{3}\EAST{2}{2}\UP{3}{3}
    \end{tikzpicture}
     \begin{tikzpicture}[scale=0.5]
        \draw[white!20] (0,0) grid (1,3);
        \node[text width=0.5cm] at (0.5, 1.5) {$+$};
        \end{tikzpicture}
   \begin{tikzpicture}[scale=0.5]
        \draw[black!20] (0,0) grid (3,3);
        \Xmarking{1}{2}\Xmarking{2}{1}\Xmarking{3}{3}
        \Asmooth{1}{3}\Asmooth{2}{3}
        \EAST{1}{1}\UP{2}{2}
    \end{tikzpicture}
        \begin{tikzpicture}[scale=0.5]
        \draw[white!20] (0,0) grid (1,3);
        \node[text width=0.5cm] at (0.5, 1.5) {$+$};
        \end{tikzpicture}
    \begin{tikzpicture}[scale=0.5]
        \draw[black!20] (0,0) grid (3,3);
        \Xmarking{1}{2}\Xmarking{2}{3}\Xmarking{3}{1}
        \Asmooth{1}{3}\EAST{2}{1}\EAST{1}{1}\UP{3}{3}\UP{3}{2}
    \end{tikzpicture}
        \begin{tikzpicture}[scale=0.5]
        \draw[white!20] (0,0) grid (1,3);
        \node[text width=0.5cm] at (0.5, 1.5) {$+$};
        \end{tikzpicture}
    \begin{tikzpicture}[scale=0.5]
        \draw[black!20] (0,0) grid (3,3);
        \Xmarking{1}{3}\Xmarking{2}{2}\Xmarking{3}{1}
        \UP{2}{3}\EAST{1}{2}\UP{3}{2}\UP{3}{3}\EAST{1}{1}\EAST{2}{1}
         \end{tikzpicture} . 
       \end{align*}  
       \caption{The process of resolving maximal crossings when $n=3$ \label{webperm}}
    \end{figure}
    
\begin{prop}[\text{See~\cite[Proposition~2.1]{HJO}}]
\label{web:uniq}
The whole process of resolving maximal crossings to get web permutations does not depend on the selection of maximal crossings. In addition,   all the resulting web permutations are distinct.
\end{prop}
Although the definition of web permutations in $\S_n$ is given from $G(id_n,\emptyset)$, we can also 
start from $G(\pi, \emptyset)$ for any non-identity permutation $\pi\in\S_n$, apply  
the same method   to resolve all maximal crossings to get some  configurations  $G(\sigma, \Cr(\sigma))$, 
and we refer to each corresponding permutation $\sigma$ as a {\bf web permutation from $\pi$} in $\S_n$.

%
%
%
%
%
%
%
%
%
%
%
%

For each $\sigma\in \S_n$,  consider the left and top boundaries of the grid  configuration $G(\sigma,\Cr(\sigma))$. 
Label the $n$ segments in the left  boundary   as $0, 1,\dots,n-1$  from bottom to top  and the $n$ segments in the top boundary   as $n,n+1,\ldots,2n-1$ from left to right.  Then the $n$ arcs in  
$G(\sigma, \Cr(\sigma))$ induce a matching of these $2n$ labels that is denoted by $M(\sigma)$. See Fig.~\ref{match} for an example when $\sigma=1324\in\Web_4$.
 
  \begin{figure}[h]
   \centering
   \begin{tikzpicture}[scale=0.5]
   
    \node[left] at (0,-1.5) {0};
   \node[left] at (0,-0.5) {1};
   \node[left] at (0,0.5) {2};
   \node[left] at (0,1.5) {3};
 
   \node[above] at (0.5,2) {4};
   \node[above] at (1.5,2) {5};
   \node[above] at (2.5,2) {6};
   \node[above] at (3.5,2) {7};
   \draw[black!20] (0,-2) grid (4,2);
   \Xmarking{4}{2}\Xmarking{2}{1}\Xmarking{3}{0}\Xmarking{1}{-1}
 
   \Asmooth{1}{0} \Asmooth{1}{1}
   \UP{3}{1}
   \EAST{2}{0}
   \Asmooth{1}{2} \Asmooth{2}{2}\Asmooth{3}{2}

    \Matching{7}{8}\Matching{9}{12}\Matching{10}{11}\Matching{13}{14}
   
   \foreach \i in {0,...,7}{
     \draw [fill] (\i+7,0) circle [radius=0.05] ;
     \node[below] at(\i+7,0) {\i};
   }
   \end{tikzpicture}
   \caption{The matching $M(1324)=\{\{0,1\},\{2,5\}, \{3,4\},\{6,7\}\}$.}
     \label{match}
 \end{figure}
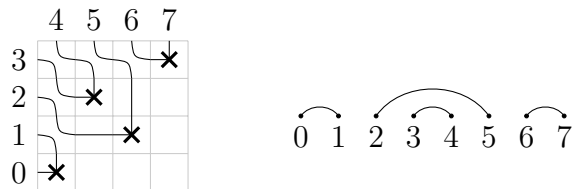
 
Recall that $\widetilde\Web_n:=\{\sigma\in\Web_n: M(\sigma)=M_0^{(n)}\}$, where $M_0^{(n)}$ is defined in~\eqref{eq:Mat0}. For example, we have $\widetilde\Web_3=\{123\}$, as only the first permutation in the bottom row of Fig.~\ref{webperm} has matching $M_0^{(3)}$. See Table~\ref{tab1} for 
permutations in $\widetilde\Web_n$ for $1\leq n\leq 6$.

\begin{table}
  \begin{tabular}{ll} 
          \hline\\[-2.5mm]
         $n=1$: &  1\\[1.8mm]
          $n=2$:& 12\\[1.8mm]
         $n=3$: &  123 \\[1.8mm]
         $n=4$:&  1234, 3412 \\[1.8mm]
        $n=5$:&12345, 14523, 34125 \\[1.8mm]
        $n=6$:&123456, 125634, 145236, 341256, 345612, 364512, 534612, 563412\\[1mm]
       \hline\\
      \end{tabular}
      \caption{Web permutations in $\widetilde\Web_n$ for $1\leq n\leq 6$.\label{tab1}}
  \end{table}

Although we can prove Conjecture~\ref{conj:hwang}, we have no idea how to characterize web permutations of $[n]$ with $M(\sigma)=M_0^{(n)}$. 
In view of Theorem~\ref{thm:huang}, we pose the following open problem. 

\begin{?}
Characterize web permutations of $[n]$ with $M(\sigma)=M_0^{(n)}$. 
\end{?}

%
%
%

\subsection{Chord diagrams}
 A set of chords of a circle having no common end-vertex is called a \emph{chord diagram} (see Fig.~\ref{AN} for two examples). 
A chord diagram is called an \emph{n-crossing} if it consists of a set of $n$ mutually crossing chords.
A 2-crossing is simply called a crossing as well. A chord diagram is called  \emph{nonintersecting} if it contains no crossing.  

Let $V$ be a set of $2n$ vertices on a circle, and let $E$ be a chord diagram consisting of  chords whose end-vertices are exactly $V$.
 In this case, $V$ is called a {\em support} of $E$.
Denote by $\mathcal{CD}(V)$ the set of all chord diagrams having $V$ as a support. Let
$x_1, x_2, x_3, x_4 \in  V$ be placed on a circle in clockwise order. Let $E \in \mathcal{CD}(V)$ and $S = \{x_1x_3, x_2x_4\}$ be a crossing in $E$. 
\emph{The expansion of $E$ with respect to $S$} is defined as a replacement of $E$ with $E_1 = (E \backslash S) \cup S_1$
or $E_2 = (E \backslash S) \cup S_2$, where $S_1 =\{x_1x_2, x_3x_4\}$ and $S_2 =\{x_2x_3, x_4x_1\} $ (see Fig.~\ref{2crossing}).
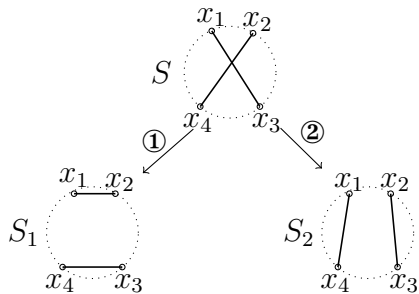
\begin{figure}[h]
\begin{tikzpicture}[scale=0.61][thick]
\draw[dotted](4,1) circle [radius=1];
\draw(4.65,0.25) circle [radius=0.06];
\draw(3.35,0.25) circle [radius=0.06];
\draw(4.5,1.85) circle [radius=0.06];
\draw(3.6,1.9) circle [radius=0.06];
\draw[semithick](4.65,0.25) - - (3.6,1.9);
\draw[semithick](3.35,0.25) - - (4.5,1.85);
\node (A) at (3.6,2.2) {$x_1$};
\node (C) at (4.6,2.1) {$x_2$};
\node (C) at (4.8,-0.1) {$x_3$};
\node (C) at (3.3,-0.1) {$x_4$};
\node (D) at (2.5,1) {$S$};
\draw[thin][-to] (3.2,-0.24) - - (2.1,-1.2);
\draw[thin][-to] (5.1,-0.22) - - (6,-1.1);
\node (D) at (2.35,-0.5) {\ding{172}};
\node (D) at (5.8,-0.3) {\ding{173}};

\draw[dotted](1,-2.5) circle [radius=1];
\draw(1.65,-3.25) circle [radius=0.06];
\draw(0.35,-3.25) circle [radius=0.06];
\draw(1.5,-1.65) circle [radius=0.06];
\draw(0.6,-1.65) circle [radius=0.06];
\node (A) at (0.6,-1.3) {$x_1$};
\node (C) at (1.6,-1.4) {$x_2$};
\node (C) at (1.8,-3.6) {$x_3$};
\node (C) at (0.3,-3.6) {$x_4$};
\node (D) at (-0.5,-2.5) {$S_1$};
\draw[semithick] (1.65,-3.25) - - (0.35,-3.25);
\draw[semithick] (1.5,-1.65) - - (0.6,-1.65);

\draw[dotted](7,-2.5) circle [radius=1];
\draw(7.65,-3.25) circle [radius=0.06];
\draw(6.35,-3.25) circle [radius=0.06];
\draw(7.5,-1.65) circle [radius=0.06];
\draw(6.6,-1.65) circle [radius=0.06];
\node (A) at (6.6,-1.4) {$x_1$};
\node (C) at (7.6,-1.4) {$x_2$};
\node (C) at (7.8,-3.6) {$x_3$};
\node (C) at (6.3,-3.6) {$x_4$};
\node (D) at (5.5,-2.5) {$S_2$};
\draw[semithick] (7.65,-3.25) - - (7.5,-1.65);
\draw[semithick]  (6.35,-3.25)- - (6.6,-1.65);
\end{tikzpicture}
 \caption{The expansion of a chord diagram with respect to $S$ (the chords
not in $S$ are not drawn). \label{2crossing}}
\end{figure}

Starting with a chord diagram $E \in \mathcal{CD}(V)$, we can get a binary tree as follows. First of all, let $E$ be the root. Arbitrarily select a crossing $S$ of $E$ and expand $E$ with respect to $S$, adding the results of two chord diagrams $E_1$ and $E_2$ as
children of $E$. Then expand $E_1$ into two chord diagrams with respect to any crossing in $E_1$ if it exists, and do the same with $E_2$.
Repeat this procedure until all chord diagrams are nonintersecting. Denote by $\mathcal{NCD}(E)$ the multiset of nonintersecting chord diagrams generated from $E$. 

\begin{lem}[See~{\cite[Lemma 1]{TN16}}]
Let $E\in \mathcal {CD}(V)$ be a chord diagram. Then beginning from $E$, the resulting multiset $\mathcal{NCD}(E)$ of nonintersecting chord diagrams is independent of the selection of the crossing at each step.
\end{lem}

For $0\leq k \leq n$, let $A(n, k)$ be a chord diagram with $n+1$ chords, where $n$-chords intersect to form an $n$-crossing, and another chord  intersects only $k$ of these $n$ chords; see Fig.~\ref{AN} (left) for an example. 
Obviously, $A(n,n)$ is simply an $(n+1)$-crossing and $A(n, 0)$ is the union of an $n$-crossing and an isolated chord. 
A chord $e$ in a nonintersecting chord diagram $F$ is called an \emph{ear} if there is no other chord inside the chord $e$ or intersects with $e$. 
If all the $n$ chords of a nonintersecting chord diagram $F$ are ears, then we call $F$ an \emph{$n$-necklace} and denote it as $N_n$, see Fig.~\ref{AN} (right).

\begin{figure}
\begin{tikzpicture}[scale=0.65][thick]
\draw[dotted](1.6,1.6) circle [radius=1.6];
\draw(2.4,3) circle [radius=0.08];
\draw(3.13,2.1) circle [radius=0.08];
\draw(3.16,1.8) circle [radius=0.08];
\draw(3.1,1.1) circle [radius=0.08];
\draw(2.75,0.5) circle [radius=0.08];
\draw(2.4,0.16) circle [radius=0.08];
\draw(1.6,0) circle [radius=0.08];
\draw(1.6,3.2) circle [radius=0.08];
\draw(0.8,3.02) circle [radius=0.08];

\draw[semithick] (2.4,3) - - (2.75,0.5);
\draw(0.095,1.1) circle [radius=0.08];
\draw[semithick] (3.13,2.1) - - (0.095,1.1);
\draw(0.03,1.4) circle [radius=0.08];
\draw[semithick] (3.16,1.8) - - (0.03,1.4);
\draw(0.1,2.1) circle [radius=0.08];
\draw[semithick] (3.1,1.1) - - (0.1,2.1);
\draw[semithick] (1.6,0) - - (1.6,3.2);
\draw[semithick] (2.4,0.16) - - (0.8,3.02);

\node (A) at (2.7,3.1) {$v_0$};
\node (A) at (3.46,2.15) {$v_1$};
\node (A) at (3.5,1.75) {$v_2$};
\node (A) at (3.45,1) {$v_3$};
\node (A) at (3.05,0.4) {$v_{4}$};
\node (A) at (2.6,-0.1) {$v_{5}$};
\node (A) at (1.6,-0.3) {$v_{6}$};

\draw[dotted](7.6,1.6) circle [radius=1.6];
\draw(7.2,3.15) circle [radius=0.08];
\draw(8,3.15) circle [radius=0.08];
\draw(6.5,2.75) circle [radius=0.08];
\draw(6.16,2.25) circle [radius=0.08];
\draw(6.14,1) circle [radius=0.08];
\draw(6.48,0.45) circle [radius=0.08];
\draw(7.25,0.05) circle [radius=0.08];
\draw(8,0.05) circle [radius=0.08];
\draw(8.7,2.75) circle [radius=0.08];
\draw(9.04,2.25) circle [radius=0.08];
\draw(9.06,1) circle [radius=0.08];
\draw(8.72,0.45) circle [radius=0.08];

\draw[semithick] (7.2,3.15) - - (8,3.15);
\draw[semithick](6.5,2.75) - - (6.16,2.25);
\draw[semithick] (6.14,1)  - -(6.48,0.45) ;
\draw[semithick] (7.25,0.05) - - (8,0.05);
\draw[semithick] (8.7,2.75) - -(9.04,2.25) ;
\draw[semithick] (9.06,1) - - (8.72,0.45);
\end{tikzpicture}
 \caption{$A(n,k)$ with $n=5$, $k=3$, and $N_n$ with $n=6$.\label{AN}}
\end{figure}
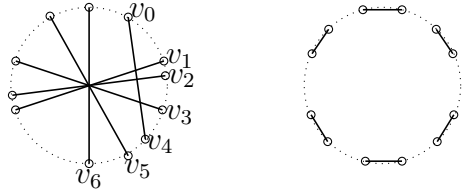

Let $v_0, v_1,... , v_{2n+1}$ be a sequence of vertices arranged in a counterclockwise direction around the circumference, and let $V = \{v_i: 0 \leq i \leq 2n + 1\}$. 
Denote by $N_{n+1,k}^+$ (resp., $N_{n+1,k}^-$) the $(n + 1)$-necklace in $\mathcal{CD}(V )$ which contains the ear $v_kv_{k+1}$ (resp., $v_{k+1}v_{k+2}$). 
Let $E$ and $F$ be two chord diagrams with a common support and $F$ is nonintersecting.  Denote by $m(E,F)$ the multiplicity of $F$ in $\mathcal{N CD}(E)$. 
For convenience, we write  $b_{n,k}^+ =m(A(n, k),N_{n+1,k}^+)$ and $b_{n,k}^- =m(A(n, k),N_{n+1,k}^-)$. 

In our proof of conjecture~\ref{conj:hwang}, we need the following results proved by Nakamigawa~\cite{TN}. 
\begin{lem}[{See~\cite[Lemma 3.2]{TN}}]\label{b+b-}
 For $1\leq k\leq n$, we have
$$b_{n,k}^-=b_{n,k-1}^+.$$
\end{lem}

\begin{thm}[{Nakamigawa~\cite[Theorem 3.1]{TN}}]\label{bs}
  For $n\geq 1$, we have
  $$\begin{cases}
   b_{2n-1,k}^+=s_{2n,n-\lfloor k/2\rfloor}\qquad \mbox{ if }~ 0\leq k\leq 2n-1;\\
   b_{2n,k}^+=s_{2n+1,\lfloor k/2\rfloor +1}\qquad \mbox{ if } ~ 0\leq k\leq 2n.
 \end{cases}$$
 \end{thm}

\subsection{Proof of conjecture~\ref{conj:hwang}}

For a given permutation $\pi\in\S_n$, 
let 
$$h(\pi):=|\{\text{web permutations}~\sigma~\text{from}~\pi: M(\sigma)=M_0^{(n)}\}|.$$
For any $1\leq k\leq n$, introduce the transformation  $p_k:\S_n\to\S_n$ by 
$$p_k(\sigma):=\sigma_2\sigma_3\cdots\sigma_k\sigma_1\sigma_{k+1}\cdots\sigma_n.$$
For example, $p_4(id_6)=234156$. Note that $p_1(\sigma)=\sigma$.  


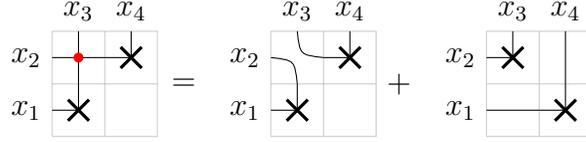
\begin{figure}
\begin{align*}
    \begin{tikzpicture}[scale=0.7]
        \draw[black!20] (0,0) grid (2,2);
        \Xmarking{1}{1}\Xmarking{2}{2}
        \Cross{1}{2}
            \node[circle, fill=red, draw=red, scale=0.3] at (0.5,1.5) {};
  \node[left] at (0,0.5) {$x_1$};
  \node[left] at (0,1.5) {$x_2$};
  \node[above] at (0.5,2) {$x_3$};
  \node[above] at (1.5,2) {$x_4$};
    \end{tikzpicture}
  &\begin{tikzpicture}[scale=0.7]
        \draw[white!20] (0,0) grid (1,3);
        \node[text width=0.5cm] at (0.6, 1) {$=$};
        \end{tikzpicture}
    \begin{tikzpicture}[scale=0.7]
        \draw[black!20] (0,0) grid (2,2);
        \Xmarking{1}{1}\Xmarking{2}{2}
        \Asmooth{1}{2}
         \node[left] at (0,0.5) {$x_1$};
  \node[left] at (0,1.5) {$x_2$};
  \node[above] at (0.5,2) {$x_3$};
  \node[above] at (1.5,2) {$x_4$};
         \end{tikzpicture}
    \begin{tikzpicture}[scale=0.7]
        \draw[white!20] (0,0) grid (1,3);
        \node[text width=0.5cm] at (0.5, 1) {$+$};
        \end{tikzpicture}
    \begin{tikzpicture}[scale=0.7]
        \draw[black!20] (0,0) grid (2,2);
        \Xmarking{1}{2}\Xmarking{2}{1}
        \UP{2}{2}\EAST{1}{1}
         \node[left] at (0,0.5) {$x_1$};
  \node[left] at (0,1.5) {$x_2$};
  \node[above] at (0.5,2) {$x_3$};
  \node[above] at (1.5,2) {$x_4$};
    \end{tikzpicture}
       \end{align*}  
       \caption{Resolve a crossing  by smoothing and switching.   \label{n=2}}
    \end{figure}

\begin{lem}\label{hlem}
For $1\leq k\leq n$, we have
$$ h(p_k(id_n))=b_{n-1,n-k-1}^+.$$
\end{lem}

\begin{proof}
 Recall that $h(p_k(id_n))$ is equal to the number of web permutations $\sigma$ from $p_{k}(id_{n})$ with $M(\sigma)=M_0^{(n)}$. 
 These web permutations $\sigma$ are obtained as follows: 
 \begin{itemize}
 \item start from the grid configuration $G(p_{k}(id_{n}),\emptyset)$ and resolve a maximal crossing at a time by smoothing or switching until all crossings are resolved; 
\item   find out the corresponding permutations $\sigma$ in the resulting configurations $G(\sigma, \Cr(\sigma))$ such that $M(\sigma)=M_0^{(n)}$.
 \end{itemize}
If we view the boundary of a grid configuration as a circumference, then resolving a crossing by smoothing (resp., switching) is equivalent to 
 expanding a crossing in a chord diagram by step~\ding{172} (resp.,~step~\ding{173}); compare Fig.~\ref{n=2} with Fig.~\ref{2crossing}. 
Moreover, the grid configuration $G(p_{k}(id_{n}), \emptyset)$ in Fig.~\ref{A(n-1,n-k-1)} (left) is in correspondence with the  chord diagram $A(n-1,n-k)$ in Fig.~\ref{A(n-1,n-k-1)} (right).
Hence each web permutation $\sigma$ from $p_{k}(id_{n})$ with $M(\sigma)=M_0^{(n)}$ corresponds to a necklace $N_n$ with a chord $01$, that is an ear $v_{n-k+1}v_{n-k+2}$ in $\mathcal{NCD}(A(n-1, n-k))$.
 Thus,  
 $$h(p_k(id_n))=m(A(n-1, n-k),N_{n,n-k}^-)=b_{n-1,n-k}^-.
 $$
 The result then follows from  Lemma~\ref{b+b-}.
\end{proof}

\begin{figure}
\begin{align*}
  \begin{tikzpicture}[scale=0.85]
  \node[left,scale=0.6] at (10.85,-1.5) {0};
  \node[left,scale=0.6] at (10.85,-0.6) {1};
  \node[left,scale=0.6] at (11,0.5) {$k$-1};
   \node[left,scale=0.65] at (10.85,1.46) {$k$};
   \node[left,scale=0.6] at (11,2.5) {$n$-2};
  \node[left] at (10.9,0.1) {$\vdots$};
  \node[left] at (10.9,2.1){$\vdots$};
  \node[left,scale=0.6] at (11,3.5) {$n$-1};
  \node[above,scale=0.6] at (11.4,4) {$n$};
  \node[above,scale=0.8] at (12,3.9) {$\cdots$};
  \node[above,scale=0.6] at (12.6,4) {$a$-2};
  \node[above,scale=0.6] at (13.5,4) {$a$-1};
   \node[above,scale=0.6] at (14.4,4){$a$};
   \node[above,scale=0.8] at (14.9,3.95) {$\cdots$};
  \node[above,scale=0.6] at (15.6,4) {$2n$-2};
  \node[above,scale=0.6] at (16.5,4) {$2n$-1};
   \node at (11.9,-0.3) {.};
  \node at (12.1,-0.1) {.};
    \node at (12.3,0.1) {.};
     \node at (14.9,1.7) {.};
  \node at (15.1,1.9) {.};
    \node at (15.3,2.1) {.};
  \draw[black!20] (11,-2) grid (17,4);
 \Xmarking{12}{0}\Xmarking{13}{1}\Xmarking{14}{-1}\Xmarking{15}{2}\Xmarking{16}{3}\Xmarking{17}{4}
   \Cross{16}{4} \Cross{15}{3}
  \Cross{15}{4} \Cross{14}{4}\Cross{13}{4} \Cross{12}{4}
   \Cross{14}{3}\Cross{13}{3} \Cross{12}{3}\Cross{12}{1}
 \UP{14}{0} \UP{14}{1} \UP{14}{2} \UP{12}{2} \UP{13}{2}
 \EAST{13}{-1}\EAST{12}{-1}\EAST{13}{2}\EAST{12}{2}\EAST{14}{2}
\end{tikzpicture}
\hspace{20mm}
&\begin{tikzpicture}[thick]
\draw[dotted](1.6,1.6) circle [radius=1.6];
\draw(2.4,3) circle [radius=0.08];
\draw(3.02,2.35) circle [radius=0.08];
\draw(3.16,1.8) circle [radius=0.08];
\draw(3.1,1.1) circle [radius=0.08];
\draw(2.4,0.16) circle [radius=0.08];
\draw(1.6,0) circle [radius=0.08];
\draw(0.8,0.18) circle [radius=0.08];
\draw(1.6,3.2) circle [radius=0.08];
\draw(0.8,3.02) circle [radius=0.08];
\draw(0.03,1.4) circle [radius=0.08];
\draw(0.1,2.1) circle [radius=0.08];
\draw(0.3,2.5) circle [radius=0.08];
\draw[thin] (2.4,3) - - (0.8,0.18);
\draw[thick,red] (3.02,2.35) - - (0.3,2.5);
\draw[thin] (3.16,1.8) - - (0.03,1.4);
\draw[thin] (3.1,1.1) - - (0.1,2.1);
\draw[thin] (1.6,0) - - (1.6,3.2);
\draw[thin] (2.4,0.16) - - (0.8,3.02);
\node [right] at (3,2.35) {$0=v_{n-k+1}$};
\node [right] at (3.2,1.85) {$1=v_{n-k+2}$};
\node[right,scale=0.9] at (3.4,1.47) {$\vdots$};
\node [right] at (3.1,0.9) {$k$-1$=v_{n}$};
\node [right] at (2.5,0.05) {$k$$=v_{n+1}$};
\node[below,scale=0.9] at (2.4,-0.1) {$\dots$};
\node [below] at (1.65,-0.1) {$n$-2};
\node [below] at (1.65,-0.5) {\tiny$\parallel$};
\node [below] at (1.8,-0.9) {$v_{2n-k-1}$};
\node [below] at (0.6,0.1) {$n$-1};
\node [below] at (0.6,-0.3) {\tiny$\parallel$};
\node [below] at (0.5,-0.7) {$v_{2n-k}$};
\node [left] at (-0.1,1.35) {$v_{2n-k+1}=n$};
\node[left,scale=0.9] at (-0.2,1.8) {$\vdots$};
\node[left] at (0.1,2.1) {$v_{2n-1}=a$-2};
\node[left] at (0.3,2.65) {$v_0=a$-1};
\node[left] at (0.87,3.2) {$v_1=a$};
\node[above,scale=0.9] at (1,3.2) {$\dots$};
\node[above] at (1.7,3.2) {$2n$-2};
\node[above] at (1.7,3.5){\tiny$\parallel$};
\node[above] at (1.7,3.9){$v_{n-k-1}$};
\node [right] at (2.5,3) {$2n$-1$=v_{n-k}$};
\end{tikzpicture}
\end{align*}
 \caption{The configuration $G(p_{k}(id_{n}),  \emptyset)$ and the diagram $A(n-1,n-k)$, where $a=n+k$.}
\label{A(n-1,n-k-1)}
\end{figure}
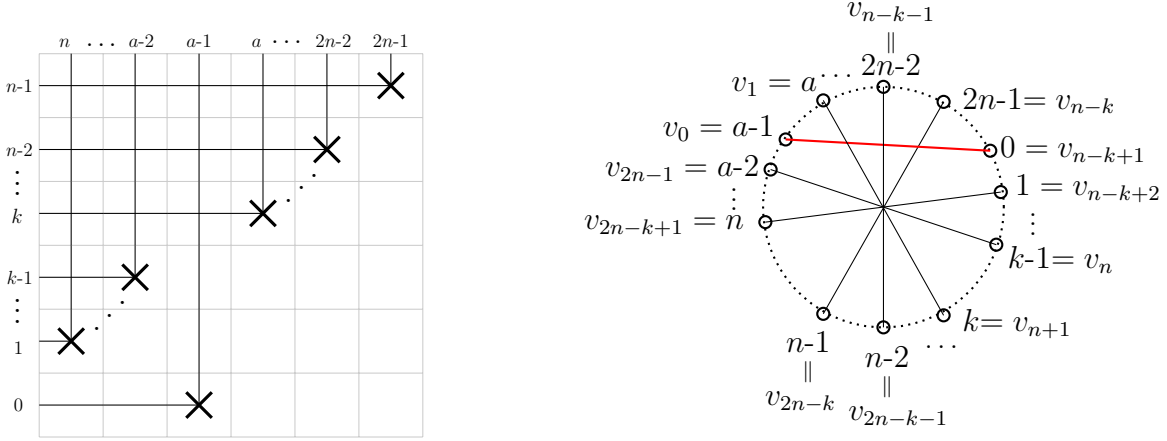

Now we are ready for the main result of this section. 

\begin{thm}\label{fb}
  For $n\geq 1$, we have
  \begin{align}\label{fbe}
  f(2n,2k-1)=b_{2n-2,2n-2k}^+,
  \end{align}
  and
  \begin{align}\label{fbo}
  f(2n-1,2k-1)=b_{2n-3,2n-2k-1}^+.
  \end{align}
  \end{thm}
  
  \begin{proof}
  Recall that $f(2n,2k-1)$ is  the number of permutations $\sigma\in \Web_{2n}$ with $M(\sigma)=M_0^{(2n)}$ and $\sigma_1=2k-1$.
  We divide the process of calculating $f(2n,2k-1)$ into two steps. 
   The first step is to start from the configuration $G(id_{2n}, \emptyset)$, 
   resolve the crossings $(1, 2n), (1, 2n-1),\ldots, (1, 2k)$ in turn by smoothing and then resolve the crossing $(1, 2k-1)$ by switching.  
  After these operations, we can obtain the grid configuration  $G(\pi, E)$, where $\pi=id_{2n}\cdot t_{1,2k-1}$ and $E=\{(1,2n),(1,2n-1),...,(1,2k)\}$. And the corresponding permutation $\pi$ satisfies $\pi_1=2k-1$.
  Since $2k-2$ and $2k-1$ have been matched, we can remove this arc (together with the $\Small{\usym{2715}}$) to get a new grid configuration $G(p_{2k-2}(id_{2n-1}),  \emptyset)$; see Fig.~\ref{f(2n,2k-1)}. 
  
  The second step is to start from $G(p_{2k-2}(id_{2n-1}),  \emptyset)$, resolve each time a maximal crossing by smoothing or switching until there is no more crossing. 
  It then follows that $f(2n,2k-1)$ equals  the number of web permutations $\tau$ from $p_{2k-2}(id_{2n-1})$ that satisfy $M(\tau)=M_0^{(2n-1)}$. That is $f(2n,2k-1)=h(p_{2k-2}(id_{2n-1}))$, which proves~\eqref{fbe} by  Lemma~\ref{hlem}.
  
   The proof of Eq.~\eqref{fbo} is similar and will be omitted. This completes the proof of the theorem. 
  \end{proof}

\begin{figure}
  \centering
  \begin{tikzpicture}[scale=0.65]
  
   \node[left,scale=0.6] at (-0.15,-1.5) {0};
  \node[left,scale=0.6] at (-0.15,-0.6) {1};
  \node[left,scale=0.6] at (0,0.5) {$2k$-3};
   \node[left,scale=0.6] at (0,1.5) {$2k$-2};
   \node[left,scale=0.6] at (0,2.5) {$2k$-1};
  \node[left] at (-0.1,0.1) {$\vdots$};
  \node[left] at (-0.1,4.1){$\vdots$};
  \node[left,scale=0.6] at (0,3.4) {$2k$};
  \node[left,scale=0.6] at (0,4.6) {2$n$-1};
  \node[above,scale=0.6] at (0.5,5) {2$n$};
  \node[above,scale=0.6] at (1.4,5) {$2n$+1};
  \node[above,scale=0.8] at (2.2,4.9) {$\cdots$};
  \node[above,scale=0.6] at (2.6,5) {$a$};
  \node[above,scale=0.6] at (3.5,5) {$a$+1};
   \node[above,scale=0.6] at (4.3,5){$a$+2};
   \node[above,scale=0.7] at (5.05,5) {$\cdots$};
  \node[above,scale=0.6] at (5.65,5) {4$n$-2};
  \node[above,scale=0.6] at (6.5,5) {4$n$-1};
  
   \node at (1.9,-0.3) {.};
  \node at (2.1,-0.1) {.};
    \node at (2.3,0.1) {.};
     
      \node at (4.9,2.7) {.};
  \node at (5.1,2.9) {.};
    \node at (5.3,3.1) {.};
        \node at (8.5,1.5) {$\rightarrow$};
  
  \draw[black!20] (0,-2) grid (7,5);
  \Xmarking{1}{2}\Xmarking{2}{0}\Xmarking{3}{1}\Xmarking{4}{-1}\Xmarking{5}{3}\Xmarking{6}{4}\Xmarking{7}{5}
  
  \Cross{6}{5}\Cross{5}{5} \Cross{4}{5}\Cross{3}{5} \Cross{2}{5}
  \Cross{5}{4} \Cross{4}{4}\Cross{3}{4} \Cross{2}{4}
   \Cross{4}{3}\Cross{3}{3} \Cross{2}{3} \Cross{2}{1}
  
  \UP{4}{0} \UP{4}{1} \UP{4}{2} \UP{2}{2} \UP{3}{2}
 \EAST{3}{-1}\EAST{2}{-1}\EAST{1}{-1}\EAST{1}{0}\EAST{1}{1}
 \Asmooth{1}{3} \Asmooth{1}{4} \Asmooth{1}{5}

 \node[left,scale=0.6] at (10.85,-1.5) {0};
  \node[left,scale=0.6] at (10.85,-0.6) {1};
  \node[left,scale=0.6] at (11,0.5) {$2k$-3};
   \node[left,scale=0.6] at (11,1.5) {$2k$-2};
   \node[left,scale=0.6] at (11,2.5) {2$n$-3};
  \node[left] at (10.9,0.1) {$\vdots$};
  \node[left] at (10.9,2.1){$\vdots$};
  \node[left,scale=0.6] at (11,3.5) {$2n$-2};

  \node[above,scale=0.6] at (11.4,4) {$2n$-1};
  \node[above,scale=0.8] at (12.1,3.9) {$\cdots$};
  \node[above,scale=0.6] at (12.6,4) {$a$-2};
  \node[above,scale=0.6] at (13.5,4) {$a$-1};
   \node[above,scale=0.6] at (14.4,4){$a$};
   \node[above,scale=0.7] at (14.9,3.95) {$\cdots$};
  \node[above,scale=0.6] at (15.6,4) {$4n$-4};
  \node[above,scale=0.6] at (16.5,4) {$4n$-3};
  
   \node at (11.9,-0.3) {.};
  \node at (12.1,-0.1) {.};
    \node at (12.3,0.1) {.};
     
      \node at (14.9,1.7) {.};
  \node at (15.1,1.9) {.};
    \node at (15.3,2.1) {.};
  
  \draw[black!20] (11,-2) grid (17,4);
 \Xmarking{12}{0}\Xmarking{13}{1}\Xmarking{14}{-1}\Xmarking{15}{2}\Xmarking{16}{3}\Xmarking{17}{4}

 \Cross{16}{4} \Cross{15}{3}
  \Cross{15}{4} \Cross{14}{4}\Cross{13}{4} \Cross{12}{4}
   \Cross{14}{3}\Cross{13}{3} \Cross{12}{3}

   \Cross{12}{1}
  \UP{14}{0} \UP{14}{1} \UP{14}{2} \UP{12}{2} \UP{13}{2}
 \EAST{13}{-1}\EAST{12}{-1}\EAST{13}{2}\EAST{12}{2}\EAST{14}{2}

  \end{tikzpicture}
  \caption{ From the configuration $G(id_{2n}\cdot t_{1,2k-1}, \{(1,2n),\ldots,(1,2k)\})$ to $G(p_{2k-2}(id_{2n-1}),  \emptyset)$, where $a=n+k$.    \label{f(2n,2k-1)}}

\end{figure}
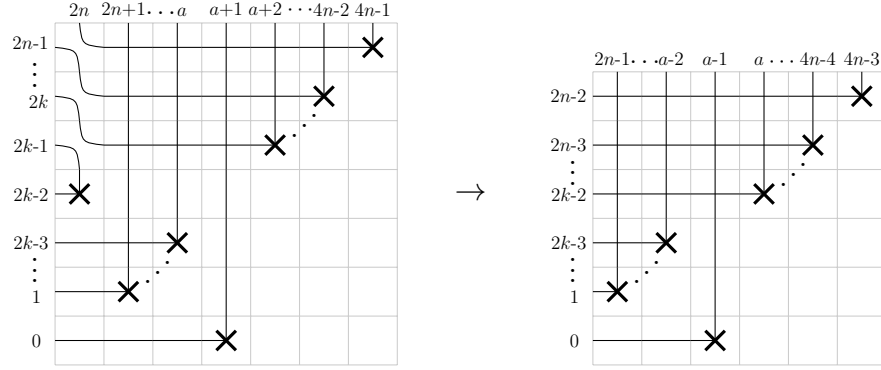

Combining Theorems~\ref{fb} and~\ref{bs} leads to a proof of Conjecture~\ref{conj:hwang}.

\section{A group action proof of~\eqref{divi:gam:at}}
\label{Sec:groact}

In this section, we construct a group action on permutations without double descents to prove combinatorially 
\begin{equation}\label{ga:XZ}
\sum_{\sigma\in\widetilde\S_n\atop{\des(\sigma)=i}}t^{\rmida(\sigma)}\alpha^{\lmi(\sigma)+\rmi(\sigma)-2}=2^i\sum_{\sigma\in\Web_{n-1}\atop{\drop(\sigma)=i}}t^{\fix(\sigma)}\alpha^{\cyc(\sigma)},
\end{equation}
thereby answering Problem~\ref{pro1:xz}.

For any $x\in[n]$ and $\sigma\in\S_n$,
the {\em$x$-factorization} of $\sigma$ is the partition of $\sigma$ into the form $\sigma=w_1w_2 xw_3w_4$,
where  $w_2$ (resp.,~$w_3$) is the maximal contiguous interval
(possibly empty) immediately to the left (resp.,~right) of $x$ whose letters are all greater than $x$.
Following Foata and Strehl~\cite{FS74} the action $\varphi_x$ is defined by 
$$
\varphi_x(\sigma)=w_1w_3xw_2w_4.
$$
For instance, with $\sigma=3\,4\,\blue{\bf 8}\,5\,\magenta{\bf 7\,10}\,1\,6\,2\,9\in\S_{10}$, in the $5$-factorization of $\sigma$,  $w_2=8$ and $w_3=7\,10$ and so $\varphi_5(\sigma)=3\,4\,\magenta{\bf 7\,10}\,5\,\blue{\bf8}\,1\,6\,2\,9$. One is also invited to check that $\varphi_4(\sigma)=3\,\magenta{\bf 8\,5\, 7\,10}\, 4\,1\,6\,2\,9$. It is plain to see that $\varphi_x$ is an involution acting on $\S_n$, and for all $x,y\in[n]$, $\varphi_x$ and $\varphi_y$ commute. Therefore, for any $S\subseteq[n]$ we can define the function $\varphi_S:\S_n\rightarrow\S_n$ by $\varphi_S=\prod_{x\in S}\varphi_x$, where multiplication is the composition of functions. Hence the group $\Z_2^n$ acts on $\S_n$ via the function $\varphi_S$. This action is called the {\em Foata--Strehl action} (FS-action for short) on permutations. Note that FS-action can be extended to permutations of a finite  set of distinct positive integers. 

We need the following characterization of Andr\'e permutations using $x$-factorization observed in~\cite{HR98}.

\begin{lem}\label{xfac:andr}
A permutation $\sigma\in\S_n$ is an Andr\'e permutation iff in the $x$-factorization $\sigma=w_1w_2xw_3w_4$ for each $x\in[n]$, $\max(w_2x)\leq\max(xw_3)$. 
\end{lem}

We also need the so-called {\em bi-basic decomposition }of permutations introduced in~\cite{DLP}. 
For $\sigma\in\S_n$ with $\lmin(\sigma)-1=k$ and $\rmin(\sigma)-1=l$, there exists a unique decomposition of $\sigma$ as
\begin{equation}\label{bi:basic}
\sigma=\blue{\a_1\a_2\cdots\a_k}1\magenta{\b_1\b_2\cdots\b_l},
\end{equation}
where 
\begin{itemize}
\item the first letter of each $\a_i$ ($1\leq i\leq k$), which is smallest  inside $\a_i$, is a left-to-right minimum of $\sigma$; 
\item  and the last letter of each $\b_i$ ($1\leq i\leq l$), which is smallest  inside $\b_i$, is a right-to-left minimum of $\sigma$. 
\end{itemize}
This is the {\em bi-basic decomposition} of $\sigma$ and each $\a_i$ or $\b_i$ is called a block of $\sigma$. An example of the bi-basic decomposition  is 
\begin{equation*}\label{bi-basic}
\sigma=\blue{{\bf3}\,4\, 8\,5\,7\,10}|1|\magenta{6\,{\bf2}}|\magenta{\bf9}.
\end{equation*}

For a word $w=w_1w_2\cdots w_n$, let $w^r:=w_nw_{n-1}\cdots w_1$ be the reversal of $w$. 
Now we can define a $\Z_2^n$-group action on $\widetilde\S_n$ to prove~\eqref{ga:XZ}. For a given permutation $\sigma\in\widetilde\S_n$ with bi-basic decomposition as in~\eqref{bi:basic} and $x\in[n]$, we define $\psi_x(\sigma)$ according to the following cases:
\begin{enumerate}
\item If $x$ is the first letter of  $\a_i$ for some $i$, then $x$ is a valley of $\sigma$. Let 
$$\tilde\a_i:=\varphi_S(\a_i),$$ 
where $S$ is the set of all double ascents of $\sigma$ inside $\a_i$.   Define $\psi_x(\sigma)$  to be the permutation  obtained from $\sigma$ by deleting the block $\a_i$ and then inserting $(\tilde\a_i)^r$ in some proper gap between blocks so that $x$ becomes an additional right-to-left minimum;  
\item If $x$ is the last letter of  $\b_i$ for some $i$, where $\b_i$ has {\bf length at least two}\footnote{In case~(1), each block $\a_i$ also has length at least two, as $\sigma$ has no double descents.},    then $x$ is also a valley of $\sigma$. Let 
$$\tilde\b_i:=\varphi_S(\b_i),$$ 
where $S$ is the set of all double ascents of $\sigma$ inside $\b_i$.
Define $\psi_x(\sigma)$ to be the permutation obtained from $\sigma$ by deleting the block $\b_i$ and then inserting $(\tilde\b_i)^r$ in some proper gap between blocks so that $x$ becomes an additional left-to-right minimum.
\item If $x\neq1$ is a valley of $\sigma$ not in the above two cases, then set  $\psi_x(\sigma)=\varphi_x(\sigma)$. 
\item Otherwise, set $\psi_x(\sigma)=\sigma$. 
\end{enumerate}
It is plain to see that $\psi_x$ is an involution acting on $\widetilde\S_n$, and for all $x,y\in[n]$, $\psi_x$ and $\psi_y$ commute. Therefore, for any $S\subseteq[n]$ we can define the function $\psi_S:\widetilde\S_n\rightarrow\widetilde\S_n$ by $\psi_S=\prod_{x\in S}\psi_x$. Hence the group $\Z_2^n$ acts on $\widetilde\S_n$ via the function $\psi_S$. We called this action the {\em block-Foata--Strehl action} (BFS-action for short). 

Let $V(\sigma)$ be the set of valleys other than $1$ in $\sigma\in\widetilde\S_n$. It is clear that $\psi_x(\sigma)=\sigma$ iff $x\notin V(\sigma)$. 

\begin{exam}
If $\sigma=\blue{{\bf3}\,4\, 8\,5\,7\,10}|1|\magenta{6\,{\bf2}}|\magenta{\bf9}\in\widetilde\S_{10}$, then $V(\sigma)=\{2,3,5\}$.  We compute that 
\begin{align*}
\psi_2(\sigma)&=3\,4\, 8\,5\, 7\,10|\cyan{\bf2\,6}|1|9,\\
\psi_3(\sigma)&=1|6\,2|\cyan{\bf4\,7\,10\,5\,8\,3}|9,\\
\psi_5(\sigma)&=3\,4\,\cyan{\bf7\,10\,5\,8}|1|6\,2|9,\\
\psi_{\{2,3\}}(\sigma)&=\cyan{\bf2\,6}|1|\cyan{\bf4\,7\,10\,5\,8\,3}|9,\\
\psi_{\{2,5\}}(\sigma)&=3\,4\,\cyan{\bf7\,10\,5\,8}|\cyan{\bf2\,6}|1|9,\\
\psi_{\{3,5\}}(\sigma)&=1|6\,2|\cyan{\bf4\,8\,5\,7\,10\,3}|9,\\
\psi_{V}(\sigma)&=\cyan{\bf2\,6}|1|\cyan{\bf4\,8\,5\,7\,10\,3}|9.
\end{align*}
\end{exam}

For $\sigma\in\S_n$,  let $\lrmi(\sigma):=\lmi(\sigma)+\rmi(\sigma)-2$. The main feature of the BFS-action lies in the following lemma, whose proof is straightforward and will be omitted.

\begin{lem}\label{lem:BFS}
 For any $\sigma\in\widetilde\S_n$ and $x\in[n]$, we have 
 $$
 (\lrmi,\rmida,\des)\,\sigma=(\lrmi,\rmida,\des)\,\psi_x(\sigma).
 $$
\end{lem}

We can now prove~\eqref{ga:XZ} combinatorially using the BFS-action.  

\begin{proof}[{\bf Combinatorial proof of~\eqref{ga:XZ}}]
For $\sigma\in\widetilde\S_n$, let $\Orb(\sigma)$ be the orbit of $\sigma$ under the BFS-action. As $\psi_S(\sigma)=\sigma$ iff $S\cap V(\sigma)=\emptyset$, we have $|\Orb(\sigma)|=2^{|V(\sigma)|}$. Assuming that $\lrmi(\sigma)=l$, then by Lemma~\ref{xfac:andr}, there exists a unique $\hat\sigma\in\Orb(\sigma)$ with bi-basic decomposition 
\begin{equation}\label{bi-rep}
\hat\sigma=1\magenta{\b_1\b_2\cdots\b_l}
\end{equation}
 satisfying  that 
 \begin{center}
($\bigstar$)\quad each block $\b_i$ when removing its last letter becomes  an Andr\'e permutation. 
 \end{center}
Thus, by Lemma~\ref{lem:BFS}, we have 
\begin{align*}
\sum_{\pi\in\Orb(\sigma)}t^{\rmida(\pi)}\alpha^{\lrmi(\pi)}&=2^{|V(\hat\sigma)|}t^{\rmida(\hat\sigma)}\alpha^{\lrmi(\hat\sigma)}\\
&=2^{\des(\hat\sigma)}t^{\rmida(\hat\sigma)}\alpha^{\rmi(\hat\sigma)-1}, 
\end{align*}
as $|V(\hat\sigma)|=\des(\hat\sigma)$ and $\lrmi(\hat\sigma)=\rmi(\hat\sigma)-1$. Summing over all orbits of $\widetilde\S_{n,i}:=\{\sigma\in\widetilde\S_n:\des(\sigma)=i\}$ under the BFS-action gives 
\begin{equation}\label{eq:BFS}
\sum_{\sigma\in\widetilde\S_{n,i}}t^{\rmida(\sigma)}\alpha^{\lrmi(\sigma)}=2^i\sum_{\hat\sigma\in\widehat\S_{n,i}}t^{\rmida(\hat\sigma)}\alpha^{\rmi(\hat\sigma)-1},
\end{equation}
where  
$$
\widehat\S_{n,i}:=\{\hat\sigma\in\widetilde\S_{n,i}:  \text{ bi-basic decomposition of $\hat\sigma$ has the form~\eqref{bi-rep} satisfying ($\bigstar$)}\}.
$$
For a word $w=w_1w_2\cdots w_k$, let $c(w)$ be the cyclic permutation $(w_1-1,w_2-1,\ldots,w_k-1)$. There is an obvious one-to-one correspondence 
$$
\hat\sigma=1\magenta{\b_1\b_2\cdots\b_l}\mapsto c(\b_1) c(\b_2)\cdots c(\b_l)=\tau
$$
  between $\widehat\S_{n,i}$ and $\{\tau\in\Web_{n-1}: \drop(\tau)=i\}$ satisfying 
  $$
  \rmida(\hat{\sigma})=\fix(\tau)\quad\text{and}\quad\rmi(\hat{\sigma})-1=\cyc(\tau). 
  $$
  The desired relationship~\eqref{ga:XZ} then follows from this bijection and~\eqref{eq:BFS}.
\end{proof}


\section{The bijection $\Lambda$ and proof of Theorem~\ref{thm:pro2xz}}
\label{Sec:pro2}
In this section, we construct the bijection $\Lambda:\Web_n\rightarrow\Delta_n$ using the the min-max trees introduced by   Hetyei and Reiner~\cite{HR98} (see also~\cite[Section~1.6.3]{St}) as an intermediate structure. 

The {\em min-max tree} $T(w)$ associated with a sequence $w=w_1w_2\cdots w_n$ of distinct integers is defined recursively as follows. Suppose that  $j$ is the least integer for which either  
$w_j=\min\{w_1,\ldots,w_n\}$ or $w_j=\max\{w_1,\ldots, w_n\}$. Then 
 define $T(w)$ to be the binary tree rooted at $w_j$ with  $T(w_1\cdots w_{j-1})$  and $T(w_{j+1}\cdots w_n)$ as left and right branches, respectively. Note that no vertex of a min-max tree $T(w)$ has only a left successor. See Fig.~\ref{MNT} (left) for the  min-max tree $T(562314)$. 
 Conversely, the left-first inorder reading of the nodes of  $T(w)$ yields the sequence $w$.
 
An interior node in $T(w)$ is called a {\em min} (resp.,~{\em max}) node if it has the minimum (resp.,~maximum) label among all its descendants.
Let $T_{w_i}$ be the subtree of $T(w)$ consisting of $w_i$ and its right branch. 
 For $1\leq i\leq n$, the {\em Hetyei--Reiner action} (HR-action for short) $\phi_i$ is an operation  on the tree $T(w)$  defined as:  
if $w_i$ is a min (resp., max) node, then replace $w_i$ with the largest (resp.,~smallest) element of $T_{w_i}$ and permute the remaining elements of $T_{w_i}$ so that they keep their same relative order. 
See Fig.~\ref{MNT} for an HR-action $\phi_2$ on $T(562314)$. 

\begin{figure}
\begin{tikzpicture}
\draw (0,0)--(-1,-1);
\draw (0,0)--(1,-1)--(2,-2);
\draw (1,-1)--(0,-2)--(1,-3);
\fill (0,0) circle (2pt);
\node[above] at (0,0) {$6$};
\node[left] at (-1,-1) {$5$};
\node[above] at (1,-1) {$1$};
\node[left] at (0,-2) {$2$};
\node[below] at (1,-3) {$3$};
\node[below] at (2,-2) {$4$};
\fill (-1,-1) circle (2pt);
\fill (1,-1) circle (2pt);
\fill (2,-2) circle (2pt);
\fill (0,-2) circle (2pt);
\fill (1,-3) circle (2pt);
\node[below] at (1,-3.5) {$T(562314)$};
\draw[][->,very thick](3.5,-1.5)--(4.5,-1.5) node[midway, above]{$\phi_2$};
\draw (7,0)--(6,-1);
\draw (7,0)--(8,-1)--(9,-2);
\draw (8,-1)--(7,-2)--(8,-3);
\fill (7,0) circle (2pt);
\node[above] at (7,0) {$1$};
\node[left] at (6,-1) {$5$};
\node[above] at (8,-1) {$2$};
\node[left] at (7,-2) {$3$};
\node[below] at (8,-3) {$4$};
\node[below] at (9,-2) {$6$};
\fill (6,-1) circle (2pt);
\fill (8,-1) circle (2pt);
\fill (9,-2) circle (2pt);
\fill (7,-2) circle (2pt);
\fill (8,-3) circle (2pt);
\node[below] at (8,-3.5) {$T(513426)$};
\end{tikzpicture}
\caption{ The HR-action  $\phi_2$ on the tree $T(562314)$.\label{MNT}}
\end{figure}
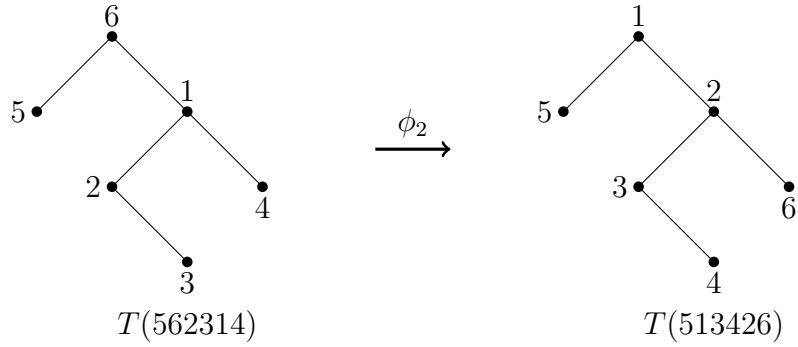

Let $T(\mathfrak{S}_n):=\{T(\sigma):\sigma\in\S_n\}$. Hetyei and Reiner~\cite{HR98}
proved that 
the action $\phi_i$ is an involution acting on $T(\mathfrak{S}_n)$ and 
$\phi_i$ and $\phi_j$ commute for all $i,j\in [n]$.
Hence, for any subset $S\subseteq [n]$ we may define the function $\phi_S:\,T(\mathfrak{S}_n)\longrightarrow T(\mathfrak{S}_n)$ by 
$$
\phi_S(T(\sigma))=\prod_{i\in S}\phi_i(T(\sigma)).
$$
The group $\mathbb{Z}_2^n$ acts on $T(\mathfrak{S}_n)$ via the functions $\psi_S$, $S\subseteq[n]$.  This action is called the HR-action on permutations. The following characterization of Andr\'e permutations is due to Hetyei and Reiner~\cite{HR98}. 

\begin{lem}[Hetyei and Reiner]\label{HR}
A permutation $\sigma\in\S_n$ is an Andr\'e permutation  if and only if all interior nodes of $T(\sigma)$ are min nodes, i.e., $T(\sigma)$ is an increasing binary tree.
\end{lem}

Let $\mathcal{A}_n$ be the set of all Andr\'e permutations in $\S_n$.   As no vertex of a min-max tree $T(\sigma)$ has only a left successor,   the following fact then follows from  Lemma~\ref{HR}.
\begin{Fact}\label{Andre}
The number of descents  of $\sigma\in\mathcal{A}_n$ equals the number of interior nodes with two children of $T(\sigma)$.
\end{Fact}

Let $\mathrm{Alt}_n$ be the set of up-down permutations in $\S_n$. 
Recall that $\Web_n$ (resp.,~$\Delta_n$) is the set of permutations of $[n]$ whose product consists  of  only Andr\'e (resp.,~up-down) cycles.
Thus, once we can establish a first-letter-preserving bijection from $\mathcal{A}_n$ to $\mathrm{Alt}_n$, applying this   bijection to each cycle of a web permutation immediately yields a bijection from $\Web_n$ to $\Delta_n$.
\vskip 3mm
\textbf{The bijection $\Lambda$ from $\mathcal{A}_n$ to $\mathrm{Alt}_n$.}
If $\sigma\in\mathcal{A}_n$, with its min-max tree $T(\sigma)$, we label the nodes  of $T(\sigma)$ alternately with the words ``min" and ``max", according to the order they appear in $\sigma$; see Fig.~\ref{Lambd} (left). Let $S$ be the set of all indices of interior nodes that receive the word ``max". For the example in Fig.~\ref{Lambd} (left), $S=\{2,4,6\}$. Define $\Lambda(\sigma)$ to be the permutation whose min-max tree is 
$\phi_S(T(\sigma))$.
Fig.~\ref{Lambd} illustrates an example of the mapping $\Lambda$.
It is clear from the construction that $\Lambda(\sigma)\in\mathrm{Alt}_n$ and that this mapping is reversible. 
Moreover, the bijection $\Lambda$ also preserves the first letter of permutations. Thus, $\Lambda$ can be extended to a bijection between $\Web_n$ and $\Delta_n$. 
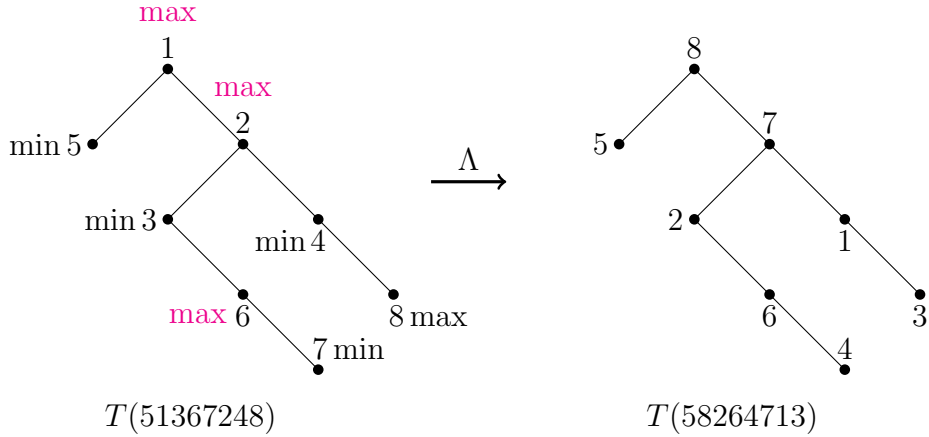
\begin{figure}
  \centering
\begin{tikzpicture}
\draw (0,0)--(-1,-1);
\draw (0,0)--(1,-1)--(2,-2)--(3,-3);\fill (3,-3) circle (2pt);\node[below] at (3.46,-3) {$8\,\max$};
\draw (1,-1)--(0,-2)--(1,-3)--(2,-4);
\fill (0,0) circle (2pt);\fill (2,-4) circle (2pt);\node[above] at (2.4,-4) {$7\min$};
\node[above] at (0,0) {$1$};
\node[above] at (0,0.5) {$\color{magenta} \max$};
\node[left] at (-1,-1) {$\min5$};
\node[above] at (1,-1) {$2$};
\node[above] at (1,-0.5) {\magenta{$\max$}};
\node[left] at (0,-2) {$\min3$};
\node[below] at (1,-3) {$6$};
\node[below] at (0.4,-3.05) {$\color{magenta} \max$};
\node[below] at (2,-2) {$4$};
\node[below] at (1.5,-2) {$ \min$};
\fill (-1,-1) circle (2pt);
\fill (1,-1) circle (2pt);
\fill (2,-2) circle (2pt);
\fill (0,-2) circle (2pt);
\fill (1,-3) circle (2pt);
\node[above] at (0.3,-5) {$T(51367248)$};
\draw[][->,very thick](3.5,-1.5)--(4.5,-1.5) node[midway, above]{$\Lambda$};
\draw (7,0)--(6,-1);
\draw (7,0)--(8,-1)--(9,-2)--(10,-3);\fill (10,-3) circle (2pt);\node[below] at (10,-3) {$3$};
\draw (8,-1)--(7,-2)--(8,-3)--(9,-4);\fill (9,-4) circle (2pt);\node[above] at (9,-4) {$4$};
\fill (7,0) circle (2pt);
\node[above] at (7,0) {$8$};
\node[left] at (6,-1) {$5$};
\node[above] at (8,-1) {$7$};
\node[left] at (7,-2) {$2$};
\node[below] at (8,-3) {$6$};
\node[below] at (9,-2) {$1$};
\fill (6,-1) circle (2pt);
\fill (8,-1) circle (2pt);
\fill (9,-2) circle (2pt);
\fill (7,-2) circle (2pt);
\fill (8,-3) circle (2pt);
\node[above] at (7.5,-5) {$T(58264713)$};
\end{tikzpicture}
\caption{ The bijection $\Lambda$ mapping $51367248$ to $58264713$.\label{Lambd}}
\end{figure}

It remains to show that 
\begin{equation}\label{eq:lambd}
\des(\sigma)=\mix(\Lambda(\sigma)). 
\end{equation}
By the recursive definition of $\mix$, we see that $\mix(\Lambda(\sigma))$ equals the number of interior nodes with two children in $T(\Lambda(\sigma))$. As the HR-action preserves the structure of the underlying  binary trees, Eq.~\eqref{eq:lambd} then follows from Fact~\ref{Andre}. It then follows from~\eqref{eq:lambd} and the definition of $\widehat{\mathrm{drop}}$ in~\eqref{def:drop} that 
$$
\mathrm{drop}(\sigma)=\widehat{\mathrm{drop}}(\Lambda(\sigma))
$$
 for $\sigma\in\Web_n$. The new interpretation of $d_{n,i}(\a,t)$ in~\eqref{int:dni2} is an immediate consequence of the bijection $\Lambda$ and~\eqref{int:dni}. This completes the proof of Theorem~\ref{thm:pro2xz}.

\vskip0.1in
For $\sigma\in\S_n$, let $\pk(\sigma)$ be the number of peaks in $\sigma$. The following result shows that the statistics ``$\pk$'' and ``$\mix$'' have the same distribution over all permutations. 

\begin{prop}\label{prop:mix}
For $n\geq1$, we have 
\begin{equation}
\sum_{\sigma\in\S_n}t^{\pk(\sigma)}=\sum_{\sigma\in\S_n}t^{\mix(\sigma)}.
\end{equation}
\end{prop}

\begin{proof}
On the one hand, Br\"and\'en~\cite[Corollary~3.2]{Br} proved that 
\begin{equation}\label{eq:Br}
A_n(t)=\sum_{i=0}^{\lfloor(n-1)/2\rfloor}2^{2i+1-n}|\{\sigma\in\S_n:\pk(\sigma)=i\}|x^i(1+x)^{n-1-2i}.
\end{equation}
On the other hand, if we let $\deg_i(T)$, $1\leq i\leq 2$, denote the number of nodes in a binary tree $T$ with $i$ children, then 
$$
\deg_1(T)=n-1-2\deg_2(T)
$$
follows by induction on the size of $T$. Then, by Fact~\#2 in~\cite[Section~1.6.3]{St} and the same discussions as the proof of~\eqref{eq:Br} in~\cite[Corollary~3.2]{Br} (with FS-action replaced by HR-action), we get
\begin{equation}\label{eq:Br2}
A_n(t)=\sum_{i=0}^{\lfloor(n-1)/2\rfloor}2^{2i+1-n}|\{\sigma\in\S_n:\deg_2(T(\sigma))=i\}|x^i(1+x)^{n-1-2i}.
\end{equation}
As $\deg_2(T(\sigma))=\mix(\sigma)$, the desired result then follows by comparing~\eqref{eq:Br} with~\eqref{eq:Br2}.   
\end{proof}

The permutation statistic ``$\pk$'' has already  been well studied in the literature (see~\cite{CF,DL,Dum2,Fu,GZ,MQYY,MY}) and it would be interesting to carry out  similar investigation on ``$\mix$'', in view of Proposition~\ref{prop:mix}. It would also be interesting to calculate the joint distribution of ``$\pk$'' and ``$\mix$'' over permutations.

\section*{Acknowledgements}

This work was supported by the National Science Foundation of China (grants 12271301, 12322115 \&  12201468) and the Fundamental Research Funds for the Central Universities.



\begin{thebibliography}{99}


\bibitem{Ang}  A. Bigeni, A generalization of the Kreweras triangle through the universal sl$_2$ weight system, J. Combin. Theory Ser. A, {\bf161} (2019), 309--326.

\bibitem{Br} P. Br\"and\'en, Actions on permutations and unimodality of descent polynomials, European J. Combin., {\bf29} (2008), 514--531.
  
\bibitem{CS} L. Carlitz and R. Scoville, Generalized Eulerian numbers: combinatorial applications, J. Reine Angew. Math., {\bf265} (1974), 110--137.

 \bibitem{CS77} L. Carlitz and R. Scoville, Some permutation problems, J. Combin. Theory Ser. A, \textbf{22} (1977), 129--145.
 
 \bibitem{CF} W.Y.C. Chen and A.M. Fu, A grammatical calculus for peaks and runs of permutations, J. Algebraic Combin., {\bf57} (2023),  1139--1162.

\bibitem{DE11}
E. Deutsch and S. Elizalde, Cycle up-down permutations, Australasian J. Combin, {\bf50} (2011), 187--199.

\bibitem{DL}Y. Dong and Z. Lin, Counting and signed counting permutations by descent-based statistics, J. Algebraic Comb., {\bf60} (2024), 145--189.

\bibitem{DLP} Y. Dong, Z. Lin and Q. Pan, A $q$-analog of the Stirling-Eulerian polynomials,  Ramanujan J., {\bf65} (2024), 1295--1311.

\bibitem{Dum2} D. Dumont, A combinatorial interpretation for the schett recurrence on the Jacobian elliptic functions, Math. Comp., {\bf33} (1979),  1293--1297.

\bibitem{Dum} D. Dumont, Further triangles of Seidel-Arnold type and continued fractions related to Euler and Springer numbers, Adv. Appl. Math., {\bf16} (1995), 275--296.

\bibitem{RE} R. Ehrenborg and E. Steingrímsson, Yet another triangle for the Genocchi numbers, European J. Combin., {\bf21} (2000),  593--600.

\bibitem{Fei} E. Feigin, Degenerate flag varieties and the median Genocchi numbers, Math. Res. Lett., {\bf18} (2011), 1163--1178.


\bibitem{F-H14}
D. Foata and G.-N. Han, Andr\'e permutation calculus: a twin Seidel matrix sequence, S\'em. Lothar. Combin., \textbf{73} (2016), Article  B73e.


\bibitem{FS-70} D. Foata and M.P. Sch\"utzenberger,
	{\em Th\'eorie g\'eom\'etrique des polyn\^omes
	eul\'eriens}, Lecture Notes in Math., Vol. 138,
	Springer, Berlin, (1970).
	

\bibitem{SF73} D. Foata and M.-P. Sch\"utzenberger, Nombres d'Euler et permutations alternantes, in A Survey of Combinatorial Theory, J.N. Srivistava, et al., eds., North-Holland, Amsterdam, 1973, pp.~173--187.

\bibitem{FS74} D. Foata and V. Strehl, Rearrangements of the symmetric group and enumerative properties of the tangent and secant numbers, Math. Z., {\bf137} (1974), 257--264.

\bibitem{Fu} A.M. Fu, A context-free grammar for peaks and double descents of permutations, Adv. in Appl. Math., {\bf100} (2018), 179--196.

\bibitem{FLS} S. Fu, Z. Lin and Z.-W. Sun, Permanent identities, combinatorial sequences, and permutation statistics, Adv. Appl. Math., {\bf163} (2025), Article 102789.

\bibitem{GZ}  I.M. Gessel and Y. Zhuang, Two-sided permutation statistics via symmetric functions, Forum Math., Sigma, {\bf12} (2024), e93: 1--37.

\bibitem{HJO} B-H. Hwang, J. Jang and J. Oh, A combinatorial model for the transition matrix between the Specht and SL$_2$-web bases, Forum of Math., Sigma, {\bf11} (2023), e82: 1--17.

\bibitem{HR98} G. Hetyei and E. Reiner, Permutation trees and variation statistics, European J. Combin., {\bf19} (1998), 847--866.

\bibitem{Ji23} K.Q. Ji, The $(\alpha,\beta)$-Eulerian Polynomials and Descent-Stirling Statistics on Permutations, \href{https://arxiv.org/pdf/2310.01053}{arXiv:2310.01053}.

\bibitem{JL} K.Q. Ji and Z. Lin,  The binomial-Stirling-Eulerian polynomials, European J. Combin., \textbf{120} (2024), Article  103962.



  \bibitem{MQYY} S.-M. Ma, H. Qi, J. Yeh and Y.-N. Yeh, On the joint distributions of succession and Eulerian statistics, Adv. in Appl. Math., {\bf162} (2025), Article 102772.
  
  \bibitem{MY} S.-M. Ma and Y.-N. Yeh, The peak statistics on simsun permutations, Electron. J. Combin.,  {\bf23} (2016), \#P2.14.

\bibitem{TN16} T. Nakamigawa, Expansions of a chord diagram and alternating permutations, Electron. J. Combin., {\bf23} (2016), \#P1.7.

\bibitem{TN} T. Nakamigawa, The expansion of a chord diagram and the Genocchi numbers, Ars Math. Contemp., {\bf18} (2020),  381--391.

\bibitem{PZ} Q. Pan and J. Zeng, Br\"and\'en's $(p,q)$-Eulerian polynomials, Andr\'e permutations and continued fractions, J. Combin. Theory Ser. A, {\bf181} (2021), Article 105445.

\bibitem{PZ2} Q. Pan and J. Zeng, Cycles of even-odd drop permutations and continued fractions of Genocchi numbers, J. Combin. Theory Ser. A, {\bf199} (2023), Article  105778. 


 \bibitem{Sei77} L. Seidel, {\"U}ber eine einfache Enstehungsweise der Bernoullischen Zahlen und einiger verwandten Reihen, Sitzungberichte der M\"unch, Akad. Math. Phys. Classe, 1877, pp.~157--187. 
 
\bibitem{St} R.P.  Stanley,  {\em Enumerative Combinatorics, vol. 1, 2nd edition}, Cambridge University Press, Cambridge, 2012.

\bibitem{Xu} C. Xu and J. Zeng, Gamma positivity of variations of $(\alpha,t)$-Eulerian polynomials, \href{https://arxiv.org/abs/2404.08470}{arXiv:2404.08470.}








\end{thebibliography}
\end{document}